\documentclass[11pt,leqno]{amsart}
\usepackage{amsmath}
\usepackage{amsfonts}
\usepackage{amssymb}
\usepackage{graphicx}
\usepackage{hyperref}
\newtheorem{theoremA}{Theorem}

\newtheorem{corollaryA}[theoremA]{Corollary}
\newtheorem{propositionA}[theoremA]{Proposition}
\newtheorem{lemmaA}[theoremA]{Lemma}

\newtheorem{theorem}{Theorem}[section]

\newtheorem{corollary}[theorem]{Corollary}

\newtheorem{definition}[theorem]{Definition}
\newtheorem{example}[theorem]{Example}

\newtheorem{lemma}[theorem]{Lemma}
\newtheorem{notation}[theorem]{Notation}
\newtheorem{problem}{Problem}
\newtheorem{proposition}[theorem]{Proposition}
\newtheorem{remark}[theorem]{Remark}

\newcommand{\inte}{\mathrm{int}}

\newcommand{\Hess}{\operatorname{Hess}}
\newcommand{\Lip}{\operatorname{Lip}}
\newcommand{\rr}{\mathbb{R}}
\newcommand{\sph}{\mathbb{S}}
\newcommand{\zz}{\mathbb{Z}}
\newcommand{\nn}{\mathbb{N}}
\newcommand{\sect}{\operatorname{Sect}}
\newcommand{\ric}{\operatorname{Ric}}

\newcommand{\vol}{\mathrm{vol}}

\newcommand{\diam}{\mathrm{diam}}

\newcommand{\scal}{\operatorname{Scal}}
\newcommand{\curv}{\mathrm{Curv}}

\newcommand{\bbb}{{\mathbb{B}}}
\renewcommand{\S}{{\mathbb{S}}}

\newcommand{\R}{\mathbb{R}}
\newcommand{\N}{\mathbb{N}}

\newcommand{\pM}{{\partial M}}
\renewcommand{\a}{{\alpha}}
\renewcommand{\b}{{\beta}}
\newcommand{\g}{{\gamma}}
\newcommand{\e}{{\epsilon}}
\newcommand{\s}{{\sigma}}

\renewcommand{\l}{{\lambda}}
\newcommand{\II}{\mathrm{II}}

\begin{document}

\title{The smooth Riemannian extension problem}
\author{Stefano Pigola}
\address{Universit\`a dell'Insubria, Dipartimento di Scienza e Alta Tecnologia\\
Via Valleggio 11, 22100 Como, Italy} \email{stefano.pigola@uninsubria.it}
\author{Giona Veronelli}
\address{Universit\'e Paris 13, Sorbonne Paris Cit\'e, LAGA, CNRS ( UMR 7539)
99\\
avenue Jean-Baptiste Cl\'ement F-93430 Villetaneuse - FRANCE} \email{veronelli@math.univ-paris13.fr}
\date{\today}

\begin{abstract}
Given a metrically complete Riemannian manifold $(M,g)$ with smooth boundary $\partial M \not=\emptyset$ and assuming that one of its curvatures is subject to a certain bound, we address the problem of whether it is possibile to realize $(M,g)$ as a domain inside a geodesically complete Riemannian manifold $(\bar M,\bar g)$ without boundary, by preserving the same curvature bounds. In this direction we provide three kind of results: (1) a general existence theorem showing that it is always possible to obtain a geodesically complete Riemannian extension without curvature constraints; (2) various topological obstructions to the existence of a complete Riemannian extension with prescribed sectional and Ricci curvature bounds; (3) some existence results of complete Riemannian extensions with sectional and Ricci curvature bounds, mostly in the presence of a convexity condition on the boundary.
\end{abstract}

\maketitle

\tableofcontents

\section*{Introduction}
Let $(M,g)$ be a given Riemannian manifold with smooth (possibly non-compact) boundary $\partial M \not = \emptyset$. This means that the Riemannian metric $g$ is a smooth, positive definite symmetric tensor up to the boundary points. Assume that $(M,g)$ is subject to some constraint on one of its Riemannian invariants, such as a curvature (or a volume growth) bound. The general problem we are interested in consists in understanding when, and to what extent, the original manifold $M$ can be prolonged past its boundary in order to obtain a new smooth Riemannian manifold $(M' ,  g')$, this time without boundary, such that one of the invariants alluded to above is kept controlled. Clearly, the most interesting situation occurs when the extended metric can be taken to be geodesically complete. In this case we can speak of $(M',  g')$ as a \textit{complete Riemannian extension of $(M,g)$ with controlled Riemannian invariants}. First insights into the possibility of constructing a complete prolongation were given by S. Alexander and R. Bishop in \cite {AB}. Actually, this paper is mostly focused on the prolongation of open manifolds without boundary, but it contains useful information also in the boundary case. The existence of a smooth extension, via gluing techniques, of compact manifolds with a strict Ricci curvature lower bound and a convexity condition on the boundary follows from work by G. Perelman, \cite{Per, Wan}. See Section \ref{section-existence-convex} below. Extensions of compact manifolds with non-negative scalar curvature up to the mean convex boundary are contained in \cite{R}. The extended metric is just $C^{2}$ but this is (abundantly) enough to get interesting rigidity results based on the positive mass theorem. Very recently, \cite{AMW}, a gluing technique in the spirit of \cite{Per} has been applied to prove that the space of metrics with non-negative Ricci curvature and convex boundary on the Euclidean three ball is path connected. In a somewhat different direction, gluing methods have been also employed by J. Wong, \cite{Won} in order to obtain isometric extensions with totally geodesic boundary and a metric-curvature lower bound in the sense of Alexandrov. This has applications to Gromov-Hausdorff precompactness results and volume growth estimates.\smallskip

In view of the well known relations between the topology of a complete Riemannian manifold and the bounds on its curvatures, or its volume growth, we are naturally led to guess that some topological obstruction appears somewhere in the extension process. In this direction, it is important to verify whether some of these obstructions are encoded in the original piece with boundary and this requires, first, a phenomenological investigation over concrete examples. For instance, a complete extension with non-negative Ricci curvature should be forbidden in general. In this respect note that the topology of a compact manifold with convex boundary and positive Ricci curvature cannot be too much wide and this is compatible with the positive results we have mentioned above; see Part \ref{section-nonexistence}. Topological obstructions should also appear at the level of upper sectional curvature bounds. Think for instance to the Cartan-Hadamard theorem, valid in the setting of geodesic metric spaces. The possibility of extending a complete simply connected manifold with boundary and negative curvature $K<0$ to a complete manifold with sectional curvature controlled by $K+\epsilon$ was addressed by S. Alexander, D. Berg and R. Bishop, \cite[p. 705]{ABB-TAMS}, during their investigations on isoperimetric properties under the assumption that the boundary has negative curvature on its concave sections. We are grateful to S. Alexander for pointing out this reference. In sharp contrast, in view of J. Lohkamp insights, \cite{Loh-Annals2}, it is expected that an upper Ricci curvature bound imposes no restrictions at all. 
\smallskip

This very brief and informal discussion serves to outline a major project concerning the systematic investigation around the  \textit{Riemannian extension problem}.
\begin{definition}
 Let $(M,g)$ be a smooth $m$-dimensional Riemannian manifold with possibly nonempty boundary. A Riemannian extension of $(M,g)$ is any smooth $m$-dimensional Riemannian manifold $(M',g')$ with possibly non-empty boundary such that $M$ is isometrically embedded in $M'$.
\end{definition}
Roughly speaking, this project could be articulated in the following problems that represent (some of) the basic steps towards a suitable understanding of the subject.

\begin{problem}[completeness]\label{problem-completeness}
Let $(M,g)$ be a metrically complete Riemannian manifold with smooth boundary $\partial M \not=\emptyset$. Does there exist a geodesically complete Riemannian extension $(M',g')$ of $M$ with $\partial M'= \emptyset$?
\end{problem}

\begin{problem}[curvature constraints]\label{problem-curvature}
 Let $\curv$ denote either of the curvatures $\sect$, $\ric$ or $\scal$ and let $C\in\R$. Let $(M,g)$ be a smooth $m$-dimensional, (non-nencessarily complete) Riemannian manifold with smooth boundary $\partial M \not=\emptyset$ satisfying $\curv_{g} < C$ (resp. $\leq C$, $>C$ or $\geq C$). Does there exist a complete, $m$-dimensional Riemannian extension $(M' ,g')$ with $\partial M' = \emptyset$ and such that the same curvature constraint holds?
\end{problem}

In the present paper we start the investigation along the lines of both these problems by presenting positive answer and obstruction results. More precisely:\smallskip

\noindent - In Part 1, we give a complete answer to Problem \ref{problem-completeness} by showing that every complete Riemannian manifold with boundary can be extended to a geodesically complete Riemannian manifold without boundary by means of a very general gluing procedure.\smallskip

In Part 2 and Part 3 we attack Problem 2 by providing both topological obstructions and existence theorems under various curvature bounds. More precisely:\smallskip

\noindent - In Part 2 we provide topological obstructions to the existence of  complete extensions satisfying $\ric \geq C$ and $\sect \leq 0$. The former are based on \v Svarc-Milnor and harmonic mappings arguments whereas the latter are obtained using both homological and homotopical methods.\smallskip

\noindent - Part 3 is devoted to the existence of complete extensions with $\ric < C$, without any assumption on the boundary, the existence of complete extensions with $\sect < C^2$ when the boundary is compact, and, in case of a compact convex boundary, existence of complete extensions under the conditions $\ric >0$, $\operatorname{Scal} >0$ and $\sect <0$.

\part{Existence of complete Riemannian extensions}
The main result of the present part of the paper states that a complete Riemannian extension can be always obtained with an amazing flexibility on the diffeomorphic class of the added piece. This is the content of the following very general theorem.

\begin{theoremA}\label{theorem-extension}
Let $(M,g_{M})$ be an $m$-dimensional connected Riemannian manifold with  smooth boundary $\pM \not= \emptyset$. Let $Q$ be any smooth $m$-dimensional differentiable manifold whose nonempty boundary $\partial Q$ is diffeomorphic to $\pM$. Then, there exists a Riemannian extension $(N,g_{N})$ of $(M,g_{M})$ such that $N\setminus M$ is diffeomorphic to the interior of $Q$. Moreover, if $(M,g_{M})$ is complete, then the extension $(N,g_{N})$ can be constructed to be complete.
\end{theoremA}
In particular, by choosing $Q=M$ in the previous statement, with the trivial identification of the boundaries, we get
\begin{corollaryA}\label{th_extension}
Let $(M,g_M)$ be a smooth complete, $m$-dimensional Riemannian manifold with smooth nonempty boundary $\partial M$. Then, there exists a geodesically complete Riemannian extension $(N,g_{N})$ of $(M,g_{M})$ with $\partial N = \emptyset$. 
\end{corollaryA}
These results are then applied in several directions. First, we observe that a given compact Riemannian manifold with boundary subject to strict curvature bounds  can be extended to a possibly incomplete Riemannian manifold with the same curvature constraints (regardless of any restriction on the geometry of the boundary); see Corollary \ref{corollary-localcurvature}. Next, we prove a density result \`a la Meyers-Serrin concerning first order Sobolev spaces on complete manifolds with boundary; see Corollary \ref{corollary-density}. Finally, as a direct consequence of Nash theorem, we observe that a complete Riemannian manifold with boundary has a proper isometric embedding into a Euclidean space; see Corollary \ref{corollary-nash}.

\section{The general gluing-deformation construction}
In this section we prove Theorem \ref{theorem-extension}. The manifolds $M$ and $Q$ are glued along the diffeomorphic boundaries and, using this ambient space, the original metric of $M$ is readily extended. At this point, the complete Riemannian extension is obtained via a careful conformal deformation. The proof that the deformed metric is actually complete relies on metric-space arguments.

\subsection{Preliminaries on metric spaces}

Given a metric space $(X,d)$, a continuous path $\gamma : [a,b] \to X$ is rectifiable if
\[
L_{d}(\gamma) := \sup \sum_{i=1}^{n} d(\gamma(t_{i-1}), \gamma(t_{i})) < +\infty
\]
where the supremum is taken with respect to all the finite partitions $t_{0}=a < t_{1} < \cdots < t_{n}=b$ of the interval $[a,b]$. In this case, the number $L_{d}(\gamma)$ is the metric-length of $\gamma$ and it is invariant by reparametrizations of the curve. On the metric space $(X,d)$ it is defined a length-distance given by
\[
d_{L} (x ,y) = \inf L_{d}(\gamma)
\]
the infimum being taken with respect to all rectifiable paths (if any) connecting $x$ to $y$. Observe that Lipschitz paths are trivially rectifiable and, conversely, every rectifiable path can be reparametrized to a constant speed, hence Lipschitz, path \cite[Proposition 2.5.9]{BBI}.
The metric space $(X,d)$ is a length metric space if $d = d_{L}$.\smallskip

Let $(M,g_M)$ be a smooth Riemannian manifold with (possibly empty) boundary $\partial M$. Its intrinsic distance, which is defined as the infimum of  the Riemannian lengths of piecewise $C^{1}$ paths connecting two given points, is denoted by $d_{(M,g_M)}$. It is well known that the metric space $(M,d_{(M,g_M)})$ is a length metric space. The Riemannian manifold $(M,g_M)$ is said to be complete if $(M,d_{(M,g_M)})$ is a complete metric space. Since $M$ is locally compact, the length-metric version of the Hopf-Rinow theorem, \cite[Theorem 2.5.28]{BBI} and Theorem \ref{th_HR} below, implies that the metric completeness of $M$ is equivalent to the Heine-Borel property which, in turn, is equivalent to the fact that a geodesic path $\gamma: [a,b) \to M$ extends continuously to the endpoint $b$. Here, by a geodesic, we mean a locally minimizing Lipschitz path. It is well known that it is $C^{1}$ regular, \cite{AA, ABB-Illinois}.\smallskip

A further notion of completeness that turns out to be very useful in applications involves the length of divergent paths. This characterization will be used to show that the glued manifold constructed in the next section is complete.

\begin{definition}
 Let $(X,d)$ be a metric space (e.g. a Riemannian manifold with possibily non-empty boundary with its intrinsic metric). A continuous path $\gamma : [a, b) \to X$ is said to be a divergent path if, for every compact set $K \subset M$, there exists $a\leq T <b$ such that $\gamma(t) \not \in K$ for every $T \leq t < b$.  The metric space $(X,d)$ is called ``divergent paths complete'' (or complete with respect to divergent paths) if every locally Lipschitz divergent path $\gamma : [0,1) \to X$ has infinite length where, clearly, $L_{d}(\gamma) = \lim_{\delta \to 1} L_{d}(\gamma|_{[0,\delta]})$.
\end{definition}

It is well known that for a manifold without boundary, the notions of metric (hence geodesic) completeness and of divergent paths completeness are equivalent.
Let us point out that a similar equivalence holds more generally on a locally compact length space
hence, in particular, for manifolds with smooth boundaries. Namely, we have the following

\begin{theorem}[Hopf-Rinow]\label{th_HR}
Let $(X,d)$ be a locally compact length space. The following assertions are equivalent.
\begin{enumerate}
\item $(X,d)$ is metrically complete, i.e. it is complete as a metric space.
\item $(X,d)$ satisfies the Heine-Borel property, i.e. every closed metric ball in $X$ is compact.
\item $(X,d)$ is geodesically complete, i.e. every constant speed geodesic $\gamma:[0,a)\to X$ can be extended to a continuous path $\bar\gamma:[0,a]\to X$
\item Every Lipschitz path $\gamma:[0,a)\to X$ can be extended to a continuous path $\bar\gamma:[0,a]\to X$
\item $(X,d)$ is divergent paths complete, i.e. every locally Lipschitz divergent path $\gamma:[0,a)\to X$ has infinite length.
\end{enumerate}
\end{theorem}

\begin{proof}
It is proven in \cite{BBI} that $(1)\Leftrightarrow(2)\Leftrightarrow(3)$. Moreover, $(4)\Rightarrow (3)$ trivially. \\
We prove that $(2)\Rightarrow (5)$. For $n\in\N$, consider the compact sets $B_n^X(\gamma(0))$. Since $\gamma$ is divergent, there exists a sequence  $\{t_n\}_{n=1}^\infty\subset (0,a)$ such that $\gamma(t_n)\not\in \overline{B}_n^X(\gamma(0))$. In particular
$$
L_d(\gamma)\geq L_d(\gamma|_{[0,t_n]})\geq d(\gamma(0),\gamma(t_n))\geq n.
$$
Since $n$ can be arbitrarily large, $\gamma$ has infinite length.\\
To conclude, we prove that $(5)\Rightarrow (4)$. Let $\gamma:[0,a)\to X$ be a Lipschitz rectifiable path. Since $\gamma$ is defined on $[0,a)$ and is Lipschitz, it has finite length. Then it can not be divergent. Namely, there exists a compact set $K\subset X$ and a sequence $\{t_n\}_{n=1}^\infty\subset (0,a)$ such that $t_n\to a$ as $n\to\infty$ and $\gamma(t_n)\in K$ for all $n$. By compactness of $K$, up to passing to a subsequence, $\gamma(t_n)\to x$ as $n\to\infty$ for some limit point $x\in K$. Set $\gamma(a)=x$. We are going to show that $\gamma:[0,a]\to X$ is continuous. Fix $\e>0$. Take $N\in\N$ large enough such that $d(\gamma(t_n),x)<\e/2$ for all $n\geq N$ and $t_N>a-\frac{\e}{2C_\g}$, where $C_\g$ is the Lipschitz constant of $\g$. Then for all $t\in(t_N,a)$,
$$
d(\gamma(t),x)\leq d(\gamma(t),\g(t_N))+d(\g(t_N),x)\leq C_\g|t-t_N|+\e/2 \leq \e.
$$
\end{proof}

We shall need to consider metric properties of curves into a manifold with boundary with respect to both the original metric and to the extended one. To this end, the following Lemma will be crucial.
\begin{lemma}\label{lemma-length}
Let $(N,g_{N})$ be a Riemannian extension of the manifold with boundary $(M,g_{M})$ and let $\gamma : [0,1] \to M$ be a fixed curve. Then
\begin{enumerate}
\item[(a)] $\gamma$ is $d_{(N,g_{N})}$-Lipschitz (resp. rectifiable) if and only if it is $d_{(M,g_{M})}$-Lipschitz (resp. rectifiable).
\end{enumerate}
Moreover, in this case:
\begin{enumerate}
 \item [(b)] $L_{g_{M}}(\gamma) = L_{g_{N}}(\gamma)$.
 \item [(c)] The speed $v_{\gamma}$ of $\gamma$, in the sense of \cite{BBI}, is the same when computed with respect to $d_{(M,g_{M})}$ and $d_{(N,g_{N})}$.
\end{enumerate}
\end{lemma}
\begin{proof}
We preliminarily observe that $d_{(N,g_{N})} \leq d_{(M,g_{M})}$ on $M$. \smallskip

(a) It is enough to consider the Lipschitz property because, as we have already recalled, every rectifiable path has a Lipschitz (constant speed) reparametrization.

We assume that $\gamma$ is $d_{(N,g_{N})}$-Lipschitz  and we prove that $\gamma$ is $d_{(M,g_{M})}$-Lipschitz, the other implication being trivial from the above observation. We shall show that, for every $t_{0} \in [0,1]$, there exists a closed interval $I_{0} \subset [0,1] $ containing $t_{0}$ in its interior such that $\gamma|_{I_{0}}$ is $d_{(M,g_{M})}$-Lipschitz.  We suppose that $\gamma(t_{0}) \in \partial M$, the other case being easier. Let $\varphi_{0} : U_{0} \to \bbb_{1}$ be a local coordinate charts of $N$ centered at $\gamma(t_{0})$ and such that $\varphi_{0}(U_{0} \cap M) = \bbb_{1}^{+}$, the upper-half unit ball. Let $V_{0} = \varphi_{0}^{-1}(\bbb_{1/2})$ and choose $I_{0}$ such that $\gamma(I_{0}) \subset V_{0}$. Note  that the distances $d_{(N,g_{N})} $ and $d_{(V_{0} ,g_{N})}$ are equivalent on $V_{0}$ and, similarly, $d_{(M,g_{M})} $ and $d_{(V_{0}\cap M ,g_{M})}$ are equivalent on $V_{0} \cap M$. Moreover, $\varphi_{0} : (V_{0},d_{(V_{0},g_{N})}) \to (\bbb_{1/2},d_{(\bbb_{1/2},g_{\mathrm{Eucl}})})$ and $\varphi_0 : (V_{0} \cap M ,d_{(V_{0} \cap M,g_{M})}) \to (\bbb^{+}_{1/2} ,d_{(\bbb^{+}_{1/2} ,g_{\mathrm{Eucl}})})$ are  bi-Lipschitz.
Since $\gamma$ is $d_{(N,g_{N})}$-Lipschitz then $\varphi_{0} \circ \gamma|_{I_{0}}$ is $d_{(\bbb_{1/2},g_{\mathrm{Eucl}})}$-Lipschitz. Since $\bbb^{+}_{1/2}$ is convex then $\varphi_{0} \circ \gamma|_{I_{0}}$ is $d_{(\bbb^{+}_{1/2},g_{\mathrm{Eucl}})}$-Lipschitz. Hence $\g|_{I_{0}}$ is $d_{(M,g_{M})}$-Lipschitz.\smallskip

(b) Using a partition of $[0,1]$ by sufficiently small subintervals we can apply \cite[Lemma 1 and Lemma 3]{AA}.\smallskip

(c) This follows from (b) and \cite[Corollary 2.7.5]{BBI}.
\end{proof}

\subsection{The proof of Theorem \ref{theorem-extension}} \label{subsection-extension-completenss}
Let $g_{Q}$ be any  Riemannian metric on $Q$ and let $\eta: \partial M \to \partial Q$ be a selected diffeomorphism. Let us consider the smooth gluing $N:= M \cup_{\eta} Q$ whose differentiable structure is obtained in a standard way using collar neighborhoods of the manifolds involved. More precisely, $N$ is the topological manifold without boundary obtained from $M \cup Q$ identifying points $x$ and $\eta(x)$ for every $x \in \partial M$. 
With a slight abuse of notation, here and on we consider $M$ and $Q$ as subsets of $N$ such that $M\cap Q=\pM$, and we identify objects on $M$ and $Q$ with their images on $N$ via the inclusions $M\hookrightarrow N$ and $Q\hookrightarrow N$.
Let $\mathcal W_M\subset M$ be an open tubular neighborhood of $\pM$ and let $p_M:\mathcal W_M\to\pM\times (-1,0]$ 
be the corresponding smooth diffeomorphism, whose restriction 
$p_M|_\pM : \pM\subset \mathcal W_M \to \pM\times 0$ is the identity map $p_M (x) = x \times 0$. 
Similarly, let $\mathcal W_Q\subset Q$ be a  tubular neighborhood of $\partial Q$ and let $p_Q: \mathcal W_Q\to\pM\times [0,1)$ 
be the corresponding smooth diffeomorphism, whose restriction 
$p_Q|_\pM : \pM\subset W_Q \to \pM\times 0$ is the identity map $p_Q (x) = x \times 0$.

Then $p_M$ and $p_Q$ induce a homeomorphism $p:\mathcal W= \mathcal W_M\cup \mathcal W_Q\subset N \to \pM\times (-1,1)$. The differentiable structure on $N$ is obtained by imposing that the homeomorphism $p$ is a smooth diffeomorphism and that the inclusions 
$j_M: \pM \hookrightarrow N$ and $j_Q: \partial Q \hookrightarrow N$ are smooth embeddings.\smallskip

The proof of Theorem \ref{theorem-extension} is now achieved in three steps that we formulate as the following Lemmas of independent interest.

\begin{lemmaA}\label{lemma-extension-1}
 Keeping the above notation, there exists a Riemannian metric $\tilde g$ on $N$ such that $\tilde g = g_{M}$ on $M$, i.e., $(N,\tilde g)$ is a Riemannian extension of $(M,g_{M})$. Moreover,  for every $\epsilon >0$, there exists a tubular neighborhood $\mathcal X_{Q}\subseteq \mathcal{W}_{Q}$ of $\pM$ in $N\setminus M$ such that:
\begin{enumerate}
 \item [(a)] $P := M \cup \mathcal{X}_{Q} \subset N$ is a manifold with smooth boundary.
 \item [(b)] there exists a $(1+\epsilon)$-Lipschitz projection $\rho:  (P,\tilde g) \to (M,g_M)$ such that $\rho|_{\mathcal X_{Q}}$ is a diffeomorphism.
\end{enumerate}
\end{lemmaA}
In what follows, the value of $\epsilon$ is irrelevant. Therefore, we will always assume that $\epsilon = 1$.
\begin{lemmaA}\label {lemma-extension-2}
 Let $(M,g_{M})$ and $(P , g_{P}= \tilde g|_{P})$ be as above. If $(M,g_{M})$ is complete then  so is $(P,g_{P})$.
\end{lemmaA}

\begin{lemmaA}\label{lemma-extension-3}
Let $(M,g_{M})$ be an $m$-dimensional  Riemannian manifold with non-empty boundary. Let $(P,g_{P})$ be a complete Riemannian extension of $M$ with non-empty boundary. Let $(N,\tilde g)$ be a Riemannian extension of $(P,g_{P})$, hence of $(M,g_{M})$. Then, there exists a Riemannian metric $g_{N}$ on $N$ such that $(N,g_{N})$ is still a Riemannian extension of $(M,g_{M})$ and it is complete.
\end{lemmaA}

The rest of the section is entirely devoted to the proofs of these results.

\begin{proof}[Proof of Lemma \ref {lemma-extension-1}]
 We proceed by steps.\smallskip

 \noindent \textit{Step 1.} First, we construct a local extension of $g_{M}$ beyond $\pM$ in $N$.
Consider on the cylinder $\pM\times(-1,1)$ a locally finite family of coordinate charts $\{(V_\b,\psi_\b): \b\in B\}$ such that 
\begin{itemize}
 \item [(i)] $\cup_{\b\in B}V_\b\supset \pM\times \{0\}$,
 \item [(ii)] $\psi_\b(V_\b)= \bbb_1$,
\end{itemize}
where $\bbb_{1}$ denotes the unit ball in the Euclidean space $\rr^{m}$. Let $\mathcal S$ be the space of symmetric $m\times m$ matrices and set 
$$\mathbb{L}_{t}=\{(x_1,\dots,x_n)\in \bbb_{1} : x_n\leq t\}.$$ 
In particular, $\mathbb{L}_{0} = \bbb_{1}^{-}$, the lower-half unit ball. Fix $\b\in B$. The metric $g_{M}$ on $p^{-1}(V_\b)
\cap M$ is represented in local coordinates by a smooth section 
$s_\b:\mathbb{L}_0\to\mathcal S$, such that $s_\b(x)$ is positive definite for all $x\in \mathbb{L}_0$. Extend smoothly $s_\b$ to a section $\tilde s_\b: \bbb_{1}  \to\mathcal S$. 
By continuity we can find a $t_\b\in (0,1]$ such that  $\tilde s_\b$ is positive definite for all $x\in \mathbb{L}_{t_\b}$. Define $\tilde V_\b=p^{-1}\circ\psi_\b^{-1}(\operatorname{int}\mathbb{L}_{t_\b})$. Repeating the construction for all $\b\in B$ we have obtained a family of local Riemannian metrics $\tilde g_\b$ defined on $\tilde V_\b$ for all $\b\in B$, such that $\tilde g_\b=g_{M}$ on $\tilde V_\b\cap M$. Moreover $\cup_{\b\in B}\tilde V_\b\supset \pM\times \{0\}$.\\

\noindent \textit{Step 2.} Next, we extend smoothly $g_{M}$ to a global metric $\tilde g$ on $N$. 
The collection of sets $\{\inte M, \inte Q,  \tilde V_\b : \b\in B\}$
gives a locally finite covering of $N$. Let $\{\eta_M,\eta_Q, \eta_\beta : \beta\in B\}$ be a subordinated partition of unity. Then 
$$\tilde g=\eta_M g_M + \eta_Q g_Q  + \sum_{\b\in B}\eta_\b\tilde g_\b$$ is a positive definite smooth Riemannian metric on $N$. Moreover, for all $x\in M$,
$$\tilde g|_x=\eta_M g_M|_x + \sum_{\b\in B}\eta_\b g|_x=g_M|_x.$$

\noindent \textit{Step 3.} Finally, we show how to construct the neighborhood $\mathcal X_Q$ and the Lipschitz projection $\rho$. 

For all $x\in\pM$, let $\nu(x)$ be the outward normal vector to $\pM$ at the point $x$. The exponential map
$\exp^\perp(x,s):=\exp_x(s\nu(x))$ is well defined for any $s$ small enough (depending on $x$), i.e. for $s\in [-s_0(x),s_0(x)]$ where we can assume that $s_0:\pM\to(0,\infty)$ is smooth. Set 
\begin{align*}
\mathcal X_Q&= \{\exp^\perp(x,s)\ :\ x\in\pM,0\leq s\leq s_0\},\\
\mathcal X_M&= \{\exp^\perp(x,s)\ :\ x\in\pM,0\geq s\geq -s_0\}.
\end{align*}
Define $\rho: M \cup \mathcal X_{Q} \to M$ as $\rho(\exp^\perp(x,s))=\exp^\perp(x,-s)$ when $s>0$ (i.e. $\rho$ reflects $\mathcal X_Q$ onto $\mathcal X_M$ with respect to Fermi coordinates) and $\rho = \mathrm{id}$ on $M$. Let $\|d \rho \|(p):= \sup_{T_{p}M \setminus \{0\}} |d_{p}\rho(v)|_{\rho(p)}/ |v|_{p}$ denotes the operator norm of $d_{p}\rho$. It is not difficult to see that  $\|d\rho \|(\exp^\perp(x,s))\to 1$ as $s\to 0$ for every $x\in\pM$, therefore we can choose the function $s_0$ so small, depending on $\epsilon$, that $\partial \mathcal X_{Q}$ is smooth and $\|d\rho \| \leq 1+ \epsilon$ on $P$. This latter bound implies that $\rho$ is a $(1+\epsilon)$-Lipschitz map. This amounts to show that, given a piecewise $C^{1}$-curve $\gamma : [0,a] \to P$, it holds
\begin{equation}\label{length}
L_{M} (\rho \circ \gamma) \leq (1+\epsilon) L_{M} (\gamma).
\end{equation}
To this aim, we note that $\rho$ is locally Lipschitz in  $P$. The only delicate points are those in the bi-collar neighborhood $\mathcal{X}_{M} \cup \mathcal{X}_{Q }$. But, in this set, $\rho$ is locally Lipschitz with respect to the product metric inherited from $\partial M \times [-1 , 1]$ and local Lipschitzianity does not depend on the ground metric. Now, the image $\rho \circ \gamma : [0,a] \to M$ is locally Lipschitz and its length satisfies
\[
L_{M} (\rho \circ \gamma) = \int_{0}^{a} v_{\rho \circ \gamma}(t) dt,
\]
where $v_{\rho\circ \gamma}$ denotes the speed of the curve in the sense of \cite{BBI}. In view of (c) of Lemma \ref{lemma-length}, since
\[
v_{\rho\circ \gamma}(t) \leq \| d \rho\|(\gamma(t)) \cdot   v_{\gamma}(t) \leq (1+\epsilon) v_{\gamma}(t)
\]
on the open and full measure subset of $[0,a]$:
\[
\gamma^{-1}(P\setminus M) \cup \inte \big([0,a] \setminus \gamma^{-1}(P \setminus M) \big)
\]
then, by integration, we deduce the validity \eqref{length}. 

\end{proof}

\begin{proof}[Proof of Lemma \ref{lemma-extension-2}]
First, we claim that given a locally Lipschitz, divergent path $\gamma : [0,1) \to P$ its (locally Lipschitz) projection $\rho \circ \gamma : [0,1) \to M$ is divergent. Indeed, if $K \subset M$ is a compact set, then $\rho^{-1}(K) = K \cup \rho|_{\mathcal X _{Q}}^{-1}(K \cap \mathcal X_{M})$ is compact in $P$. Therefore, there exists $0 \leq T <1$ such that $\gamma(t) \not \in \rho^{-1}(K)$ for every $T \leq t <1$. It follows that $\rho \circ \gamma (t) \not \in K$ for $T \leq t <1$, proving the claim.

Now, by Theorem \ref{th_HR}, $(M,g_{M})$ is divergent paths complete and therefore $L_{g_{M}}( \rho \circ \gamma) = +\infty$. Since $\rho$ is $2$-Lipschitz, we conclude that $L_{g_{P}}(\gamma) = +\infty$, as desired.
\end{proof}

\begin{proof}[Proof of Lemma \ref{lemma-extension-3}]
 Consider an exhaustion of $N$, i.e. a sequence $\{N_j\}_{j=0}^\infty$ of compact manifolds with smooth boundary such that $N_j\Subset N_{j+1}\subset N$ for all $j\geq0$ and $\cup_{j=0}^\infty N_j=N$. In the following, we use the convention $N_j=\emptyset$ whenever $j<0$. Call:
 \begin{itemize}
\item  $N_{j,a}$ any connected component of $(N\setminus \inte P)\cap (\overline{N_{j+1}\setminus N_{j}})$ for $a\in A_j$;
\item $\hat N_{j,b}$ any connected component of $(N\setminus \inte P)\cap(\overline{N_{j+2}\setminus N_{j-1}})$ for $b\in B_j$. Observe that $\#B_j\leq \# A_j<\infty$ for all $j$.
\end{itemize}
Finally, define
\begin{itemize}
\item $\partial_P\hat N_{j,a}=\hat N_{j,a}\cap \partial P$.
\end{itemize}
 
We have the following

\begin{lemma}\label{lem_metric}
There exists a smooth Riemannian metric $g_{N}$ on $N$ such that $(N,g_{N})$ is a Riemannian extension of $(M,g_{M})$ and, for all $j\in\N$, $a\in A_j$ and $b\in B_j$, the following hold:
\begin{enumerate}
\item [(a)] Let $x,y\in N_{j,a}$ with $x\in \partial N_j$ and $y\in \partial N_{j+1}$.  If $\gamma : [0,1] \to N_{j,a}$ is any Lipschitz path connecting $x$ to $y$ then $L_{g_{N}}(\gamma) \geq 1$.
\item [(b)] Let $x,y \in \partial_P\hat N_{j,b}$.
If $\gamma : [0,1] \to \hat N_{j,b}$ is any Lipschitz path connecting $x$ to $y$ then $L_{g_{N}}(\gamma) \geq d_{(P,g_{P})}(x,y)$.
\end{enumerate}
\end{lemma}

\begin{proof}
For the ease of notation, given a subset $C$ of $(N,\tilde g)$ we shall denote by $\tilde d_{C}$ the length metric on $C$ induced by $(N,\tilde g)$, namely,
\[
\tilde d_{C}(c_{1},c_{2}) = \inf L_{\tilde g}(\gamma)
\]
where the infimum is taken over the Lipschitz path in $C$ (if any) connecting $c_{1}$ with $c_{2}$.\smallskip

For any  $j\in\N$ and $a\in A_j$, define
\[
q_1^{j,a}=\inf \tilde d_{N_{j,a}}(x,y),
\]
where the infimum is taken over all the $x,y\in N_{j,a}$ such that $x\in \partial N_j$ and $y\in \partial N_{j+1}$. Since $\partial N_{j+1}\cap (N\setminus \inte P)$ and $\partial N_j\cap (N\setminus \inte P)$ are  compact and disjoint, $q_1^{j,a}>0$.\\
For any $j\in\N$ and $b\in B_j$, define
$$
 \delta^{j,b}(x,y)= \frac {\tilde d_{\hat N_{j,b}}(x,y)}{d_{(P,g_{P})}(x,y)},
$$
and
\[
q_2^{j,b}=\inf  \delta^{j,b}(x,y),
\]
where the infimum is taken over all the $x\neq y$ belonging to $\partial_P\hat N_{j,b}$. We claim that $q_2^{j,b}>0$. Indeed, suppose $q_{2}^{j,b}=0$. Then, there exist sequences of points $\{x_{k}\}$ and $\{y_{k}\}$ in $\partial_{P} \hat N_{j,b} \subset \partial P$  such that $\delta^{j,b}(x_{k},y_{k}) \to 0$. Since $\tilde d_{\hat N_{j,b}} \geq d_{(N,\tilde g)}$ on $\hat N_{j,b}$, we deduce that
\[
\frac { d_{(N,\tilde g)}(x_{k},y_{k})}{d_{(P,g_{P})}(x_{k},y_{k})} \to 0.
\]
Since $x_{k},y_{k}$ are in a compact subset of $P$ then the denominator  $d_{(P,g_P)}(x_{k},y_{k})$ is uniformly bounded. It follows that $d_{(N,\tilde g)}(x_{k},y_{k}) \to 0$. Therefore, by compactness of $\partial_{P} \hat N_{j,b}$, and up to passing to subsequences, we can assume that $\{x_{k}\} , \{y_{k}\}$ converge to a same point $z \in \partial_{P} \hat N_{j,b}$ with respect to the $d_{(N,\tilde g)}$ metric.
Since $P$ is a manifold with smooth boundary,
 \[
 \frac { d_{(N,\tilde g)}(x_{k},y_{k})}{d_{(P,g_{P})}(x_{k},y_{k})}= \frac { d_{(N,\tilde g)}(x_{k},y_{k})}{d_{(P,\tilde g)}(x_{k},y_{k})} \to 1,
 \]
 a contradiction.\\

For every $j\in\N$, $a\in A_j$ and $b\in B_j$, let $\mu_{j,a},\nu_{j,b}\in C^{\infty}_c((N\setminus M)\cap N_{j+2})$ be such that $0\leq\mu_{j,a},\nu_{j,b},\leq 1$, 
$$\mu_{j,a}|_{N_{j,a}}\equiv 1,\ \mu_{j,a}|_{N_{j-1}}\equiv 0,\ \nu_{j,b}|_{\hat N_{j,b}}\equiv 1,\ \nu_{j,b}|_{ N_{j-2}}\equiv 0.
$$
We define the smooth Riemannian metric $g_{N}$ on $N$ as
\begin{align*}
g_{N}(x)=e^{2\sum_{j=0}^\infty\left[\sum_{a\in A_j}\max\{0;-\ln(q_1^{j,a})\}\mu_{j,a}(x)
+\sum_{b\in B_j}\max\{0;-\ln(q_2^{j,b})\}\nu_{j,b}(x)\right]}
\tilde g(x).
\end{align*}
Note that $g_{N}$ is well defined, since the sum is locally finite. Moreover the conformal factor is everywhere greater or equal to $1$, and it is greater or equal to $(q_1^{j,a})^{-2}$ on $N_{j,a}$ and to  $(q_2^{j,b})^{-2}$ on $\hat N_{j,b}$. So the metric $g_{N}$ satisfies the claim of the lemma.
\end{proof}

To conclude the proof of Lemma \ref{lemma-extension-3}, we have to show that the metric $g_{N}$ of $N$ obtained in Lemma \ref{lem_metric} is (divergent paths) complete. To this end, we take a locally Lipschitz divergent path $\gamma :[0,1) \to N$  and we distinguish three different cases:\\

\noindent \textit{First case}. The path $\gamma$ is definitely contained in $N \setminus P$. Without loss of generality we can assume that the entire path $\gamma$ is contained in $N \setminus P$. Using item (a) of Lemma \ref{lem_metric} we easily deduce that $L_{g_{N}}(\gamma) = +\infty$.\\

\noindent \textit{Second case}. The path $\gamma$ is definitely contained in $\inte P$. As above, we can assume that $\gamma$ is entirely in $P$. Then, by assumption, $L_{g_{P}}(\gamma) = +\infty$. On the other hand, by definition of $g_N$ we have that  $L_{g_{N}} \geq L_{g_{P}}$ and, therefore, $L_{g_{N}}(\gamma) = +\infty$.\\

\noindent \textit{Third case}. There exists a sequence of times $t_{k} \to 1^{-}$ such that $\gamma (t_{2k}) \in N \setminus  P$ and $\gamma(t_{2k+1}) \in \inte P$ for all $k$. By contradiction, let us assume that $L_{g_{N}}(\gamma) < +\infty$. Then, up to starting from $T$ close enough to $1$ we can assume that $\ell := L_{g_{N}}(\gamma) < 1$ and that $\gamma(0) \in P$.
We consider the natural reparametrization of $\gamma$ and we assume that $\gamma : [0,\ell) \to N$ has unit speed; \cite[Proposition 2.5.9]{BBI}.

Consider the disjoint union 
\[
\gamma^{-1}(N \setminus P) = \dot \cup_{\lambda \in \nn} (\a_{\l}, \b_{\l}).
\]
Then, by item (a) of Lemma \ref{lem_metric}, for each $\l$ there exist $j_{\l}\in \nn$ and $b_{\l} \in B_{j_{\l}}$ such that
\[
\gamma((\a_{\l},\b_{\l})) \subset \hat N_{j_{\l},b_{\l}}.
\]
By item (b) of Lemma \ref{lem_metric}, for every $\l$,
\[
d_{(P,g_{P})}(\gamma(\a_{\l}), \gamma(\b_{\l})) \leq L_{g_{N}} (\g|_{(\a_{\l},\b_{\l})}).
\]
Hence there exists a Lipschitz curve $\s_{\l} : [\a_{\l},\b_{\l}] \to P$ with the same endpoints of $\gamma|_{[\a_{\l},\b_{\l}]}$, i.e.,
\[
\s_{\l}(\a_{\l}) = \gamma(\a_{\l}),\quad \s_{\l}(\b_{\l}) = \gamma(\b_{\l}),
\]
and such that
\begin{equation}\label{length-g}
L_{g_{P}}(\s_{\l}) \leq 2 L_{g_{N}} (\g|_{(a_{\l},\b_{\l})}) = 2( \b_{\l} - \a_{\l} ).
\end{equation}
We now construct a new path $\sigma : [0,\ell) \to P$ by setting
\[
\sigma(t) = \begin{cases}
\s_{\l}(t) &\text { if } t\in (\a_{\l},\b_{\l}), \text{ for some }\l \in \nn \\
\gamma(t) &\text {otherwise. }
\end{cases}
\]
Set $\mathcal A_n:=\cup_{\l=0}^n(\a_\l,\b_\l)$. For every $n\in\N$ we introduce the $d_{(N,\tilde g)}$-rectifiable paths $\gamma_n:[0,\ell)\to N$ by
\[
\gamma_n(t)=\begin{cases}
\s_{\l}(t) &\text { if } t\in \mathcal A_n,\\
\gamma(t) &\text {otherwise. }
\end{cases}
\]
From \eqref{length-g}, item (b) of Lemma \ref{lemma-length}, and the fact that, by construction, lengths with respect to $\tilde g$ are smaller than lenghts with respect to $g_N$, we deduce that for all $n\in \N$,
\begin{align*}
L_{\tilde g}(\g_n)&=L_{\tilde g}(\g|_{[0,\ell)\setminus \mathcal A_n})+ \sum_{i=0}^n L_{\tilde g}(\s_\l|_{(a_\l,\b_\l)})\\
&= L_{\tilde g}(\g|_{[0,\ell)\setminus \mathcal A_n})+ \sum_{i=0}^n L_{g_P}(\s_\l|_{(a_\l,\b_\l)})\\
&\leq L_{g_N}(\g|_{[0,\ell)\setminus \mathcal A_n})+ \sum_{i=0}^n 2L_{g_N}(\g|_{(a_\l,\b_\l)})\\
&\leq 2L_{g_N}(\g)=2\ell.
\end{align*}
By the semi-continuity of $L_{\tilde g}$ we get that $\s$ is $d_{(N,\tilde g)}$-rectifiable, and 
\begin{equation}\label{finite-length}
L_{g_P}(\s)=L_{\tilde g}(\s)\leq 2\ell.
\end{equation}
Namely, for any fixed $S\in(0,\ell)$ and for any finite partition
$0=s_0<s_1<\dots<s_K=S$, 
there exists $n\in N$ such that $\gamma_n(s_j)=\sigma(s_j)$ for all $j=0,\dots,K$, so that
\[\sum_{j=1}^K d_{(N,\tilde g)}(\s(s_{j-1}),\s(s_{j}))=\sum_{j=1}^K d_{(N,\tilde g)}(\gamma_n(s_{j-1}),\gamma_n(s_{j}))\leq L_{\tilde g}(\g_n)\leq 2\ell.
\]
Finally we show that $\s$ is divergent in $(P,g_P)$. This fact, together with \eqref{finite-length} will contradict the divergent paths completeness of $(P,g_P)$, thus concluding the proof of Lemma \ref{lemma-extension-3}.

To this purpose, fix a compact $C\subset P$ and let $j$ be large enough so that $C\subset N_j$. Since $\g$ is divergent in $N$, there exists $T\in [0,\ell)$ such that $\g(t)\not \in N_{j+1}$ for all $t\in (T,\ell)$. Set 
\[
\mathcal T:=\{\l\in\N\ :\ a_\l>T\text{ and }\s_\l([\a_\l,\b_\l])\cap N_j\neq \emptyset\}.\]
If $\mathcal T$ is empty, there is nothing to prove. Otherwise note that, for every $\l\in\mathcal T$, $\sigma_\l(\a_\l)\not \in N_{j+1}$ and $\sigma_\l(\b_\l)\not \in N_{j+1}$. Define
\[
c_j:=\min\{d_{(N,\tilde g)}(x,y)\ :\ x\in \partial N_j\text{ and }y\in \partial N_{j+1}\},
\]
which is well defined by compactness, and strictly positive since $N_j\Subset N_{j+1}$.
Then 
\[
\sharp\mathcal T \leq  \frac{L_{g_p}(\s)}{
2c_j}<\infty.
\]
Accordingly, we have that $\b^\ast:=\max_{\l\in\mathcal T}b_\l$ satisfies $\b^\ast<\ell$ and $\s([\b^\ast,\ell))\subset N\setminus N_{j}\subset N\setminus C$.

\end{proof}

\section{Some applications}

According to Theorem \ref{theorem-extension}, a Riemannian manifold with smooth boundary can be always realized as a smooth, closed domain of a Riemannian manifold without boundary. Moreover, the ambient manifold can be chosen to be geodesically complete if the original manifold with boundary was metrically complete (hence a closed domain). This viewpoint on manifolds with boundary has two main consequences: on the one hand, open relations concerning Riemannian quantities of local nature extend trivially past the boundary of the manifold. On the other hand, by restriction, one can easily inherit basic results and constructions from complete manifolds without boundary.  We shall provide examples of both these instances.

\subsection{Local extensions with curvature constraints}
Let $(M,g_{M})$ be a Riemannian manifold with boundary $\partial M \not= \emptyset$ satisfying a strict curvature condition like $\curv_{M}>C$ or $\curv_{M} < C$ for some constant $C \in \rr$. Here, $\curv$ denotes either the sectional, the Ricci or the scalar curvature of the manifold at hand.\smallskip

Consider any Riemannian extension $(N,g_{N})$ of $(M,g_{M})$. Since $\curv >C$ (resp. $\curv < C$) and $\curv_{M} = \curv_{N}$ on $M$, by continuity there exists a neighborhood $U \subseteq N$ of $\partial M$ such that $\curv > C$ (resp. $\curv < C$) holds on $V = M \cup U$. \smallskip

Assume now that $\partial M$ is compact. We say that $\partial M$ is strictly convex (resp. strictly concave) if, with respect to the outward pointing Gauss map $\nu$, the eigenvalues $\lambda_{1},\cdots,\lambda_{m-1}$ of the shape operator $\mathcal{S}(X) = -\text{ }^{N}\!D_{X}\nu$ satisfy $\lambda_{j} <0$ (resp. $>0$). We choose $0 < \delta \ll 1$ in such a way that the normal exponential map $\text{ }^{N}\!\exp^{\perp}: \partial M \times (-\delta , \delta) \to V$
defines e diffeomorphism onto its image  and we can consider the corresponding family of  (diffeomorphic) parallel hypersurfaces
\[
(\partial M)_{t} = \text{ }\!^{N}\!\exp^{\perp}(\partial M \times \{t\}).
\]
Let $\mathcal{S}_{t}$ denote the shape operator of $(\partial M)_{t}$. It is known that its eigenvalues $\lambda_{1}(t),\cdots,\lambda_{m-1}(t)$ evolve (for a.e. $t$) according to the Riccati equation
\[
\frac{d \lambda_{j}}{dt} (t) = \lambda_{j}^{2} (t) + \sect_{N} (\nu \wedge E_{j}(t))
\]
where $E_{j}(t) \in T (\partial M)_{t}$ is the eigenvector of $\mathcal{S}_{t}$ corresponding to $\lambda_{j}(t)$; see e.g. \cite{G}. From this equation, under curvature restrictions and using comparison arguments, one could obtain sign conclusions on suitable intervals. Anyway, regardless of any curvature assumption, if
\[
\lambda_{j}(0) = \lambda_{j} <0
\]
(resp. $>0$), by continuity we find $0<\epsilon < \delta$ such that
\[
\lambda_{j}(t) < 0,\quad 0 \leq t \leq \epsilon,
\]
(resp. $>0$). Clearly, similar considerations hold for the mean curvature function. Thus, by taking $\bar N := M \cup \text{ }\!^{N}\!\exp^{\perp}(\partial M \times [0,\epsilon])$, we have proved the following result.

\begin{corollary}\label{corollary-localcurvature}
 Let $(M,g_{M})$ be a Riemannian manifold with boundary $\partial M \not= \emptyset$ and satisfying $\curv_{M} > C$ (resp. $\curv_{M} <C$). Then, there exists a Riemannian extension $(\bar N,g_{\bar N})$ of $M$ such that $\curv_{\bar N} > C$ (resp. $\curv_{\bar N} <C$). Moreover, assume that $\partial M$ is compact. If $\partial M$ is either strictly (mean) convex  or strictly (mean) concave, then $\bar N$ can be chosen so to have a boundary $\partial \bar N$ with the same property.
\end{corollary}
\begin{remark}
 \rm{
The situation is significantly more difficult if either we replace the strict inequalities with their weak counterparts or if we insist that the extended manifold is complete. In these cases, smooth extensions are not allowed in general. Results and examples along the second mentioned direction will be presented in Part \ref{part-nonexistence} and Part \ref{part_existence} of the paper.
}
\end{remark}

\subsection{Sobolev spaces}
In the geometric analysis on manifolds with boundary, the theory of (first order) Sobolev spaces, and the corresponding  density results, are vital to carry out PDE's constructions typical of the setting of manifolds without boundary. By way of example, we can mention the truncation method in order to obtain sub(super) solutions of Neumann problems for the Laplace operator and its applications to potential theory; see \cite{IPS}. In this respect, the Euclidean arguments work almost verbatim once we consider the manifold with boundary as a domain inside an ambient manifold without boundary. We are going to illustrate quickly this viewpoint by recovering a classical density result \'a la Meyers-Serrin; see e.g. \cite[Appendix A]{IPS}.\smallskip

Let $(M,g_{M})$ be a (possibly non-compact and incomplete) Riemannian manifold with boundary $\partial M \not = \emptyset$. Since $\inte M$ is a smooth manifold without boundary we can define, as usual, the space
\[
W^{1,p}(\inte M) = \{u:\inte M \to \rr: u \in L^{p}, \nabla u \in L^{p}\},
\]
where $\nabla u$ is the distributional gradient of $u$, endowed with the norm
\[
\| u \|_{W^{1,p}} = ( \| u \|^{p}_{L^{p}} + \| \nabla u \|^{p}_{L^{p}} )^{1/p}.
\]
Suppose now that $(M,g_{M})$ is complete and let $(N,g_{N})$ be a geodesically complete Riemannian extension without boundary. Fix a  locally finite, relatively compact, smooth atlas $\{(V_{j} , \varphi_{j})\}$ of $N$ such that either $V_{j} \cap M= \emptyset$ or $(V_{j} \cap M , \varphi_{j}|_{M})$ is a smooth chart of $M$. Without loss of generality, we can assume that $\varphi_{j}(V_{j}) = \bbb_{1} \subset \rr^{m}$ and (in case $V_{j} \cap \pM \not= \emptyset$) $\varphi_{j} (V_{j} \cap M) = \bbb^{+}_{1}$. We consider a partition of unity $\{\chi_k\}$ subordinated to the covering $\{V_k\}$ and, given a function $u \in W^{1,p}(\inte M)$, we decompose it as  $u= \sum_k u_k$ with $u_k = u \cdot \chi_k$. Now, for any fixed $\varepsilon >0$, applying in local coordinates the standard approximation procedure, e.g. \cite[Theorem 10.29]{Le}, we find $v_k \in C_{c}^{\infty}(U_k)$ such that
\[
\| u_k  - v_k\|_{W^{1,2}(\inte M)} \leq \frac{\varepsilon}{2^k}.
\]
Thus, the locally finite sum $v =\sum_k v_k$ is a function in $C^{\infty}(N)$ and gives an $\varepsilon$-approximation of $u$ in the space $W^{1,2}(\mathrm{int}M).$ This implies the partial result:
\begin{equation}\label{MS2}
 W^{1,p}(\inte M) = \overline{C^{\infty}( M) }^{\| \cdot \|_{W^{1,p}}}.
\end{equation}
Finally, we have to approximate $v|_{M}$ in $W^{1,p}(\inte M)$ with the restriction to $M$ of a function in $C^{\infty}_{c}(N)$. To this end take your favorite smooth function $\rho_{N} : N \to \rr_{>0}$ satisfying $\rho_{N} (\infty) = +\infty$ and $\| \nabla \rho_{N} \|_{L^{\infty}(N)} \leq L$. It can be obtained by regularizing the distance function by convolution methods; \cite{GW}. Moreover, choose  $\psi : \rr \to [0,1]$ to be any smooth function such that $\psi(t)  = 1$ if $t \leq 1$ and $\psi(t) = 0$ if  $t \leq 2$, and define the sequence $\psi_{k} := \psi (\rho_{N}/k)\in C^{\infty}_{c}(N)$. Then, $ \psi_{k}  \to 1$,  as $k \to +\infty$ uniformly on compact subsets of  $N$, and $\|\nabla \psi_{k}\|_{L^{\infty}(N)} \to 0$   as $k\to +\infty$.
It is then obvious, by dominated convergence, that  the sequence
\[
w_{k} = v \cdot \psi_{k}\in C^{\infty}_{c}(N)
\]
converges in $W^{1,p}(N)$ to $v$. By restriction, $w_{k}|_{M} \in C^{\infty}_{c}(M)$ converges to $v|_{M}$ in $W^{1,p}(\inte M)$. We have thus obtained the stronger density result:
\begin{corollary}\label{corollary-density}
Let $(M,g_M)$ be a complete Riemannian manifold with (possibly empty) boundary $\partial M$. Then
\begin{equation}\label{MS3}
  W^{1,p}(\inte M) = \overline{C_{c}^{\infty}( M) }^{\| \cdot \|_{W^{1,p}}}.
\end{equation}
\end{corollary}
As a side product, observe that by taking $\rho_{M} = \rho_{N}|_{M}$ we  also obtain the existence of a smooth, globally Lipschitz, exhaustion function on any complete manifold with boundary.
\begin{lemma}\label{lem_exhaustion}
 Let $(M,g_{M})$ be a complete Riemannian manifold with boundary $\partial M \not= \emptyset$. Then, there exists a smooth function $\rho_{M} : M \to \rr_{>0}$  satisfying
\begin{equation}\label{exhaustion}
\rho_{M} (\infty) = +\infty; \quad \| \nabla \rho_{M} \|_{L^{\infty}(M)} \leq L.
\end{equation}
\end{lemma}
The proof of this fact is not completely obvious if we use the pure viewpoint of manifolds with boundary. The mollification procedure used to regularize a given Lipschitz function (e.g. the intrinsic distance function) requires some care.

\subsection{Proper Nash embedding} The classical formulation of the Nash embedding theorem states that any Riemannian manifold $(N,g_{N})$ can be isometrically embedded in some Euclidean space $\rr^{\ell}$, where $\ell = \ell(\dim N)$. In the ``survey'' part of the paper \cite{GR} it is claimed that the embedding can be chosen to be proper if $N$ is geodesically complete and, moreover, that the Nash embedding holds also for manifolds with boundary. An elementary, but clever, proof of the first claim can be found in \cite{Mul}. Here, we point out that the second claim can be trivially deduced from the first one,  by restricting to the manifold with boundary a proper isometric embedding of a complete Riemannian extension. Alternatively, we can adapt the direct argument in \cite{Mul} to the case of non-trivial boundary.
\begin{corollary}\label{corollary-nash}
Let $(M,g_{M})$ be a complete Riemannian manifold with boundary $\partial M \not= \emptyset$. Then, there exists a proper isometric embedding of $M$ into some Euclidean space  $\rr^{\ell}$ where $\ell = \ell(\dim M)$.
\end{corollary}
\begin{proof}
We preliminarly observe that any Riemannian manifold with boundary has an isometric embeddings into some Euclidean space $\rr^{n}$. This follows by applying the usual Nash embedding to a Riemannian extension.

Now, let $\varrho : M \to \rr_{>0}$ be the exhaustion function of Lemma \ref {lem_exhaustion}. Up to a dilation  we can assume that $\| \nabla \varrho \|_{\infty} \leq 1/2$.  Define on $M$ the new Riemannian metric $\tilde g = g - d\varrho \otimes d\varrho$. Then, there exists an isometric embedding  $j:(M,\tilde g) \hookrightarrow \rr^{n}$, for some $n$.  It follows that $i = (j,\varrho):(M,g) \hookrightarrow \rr^{n+1}$ is a proper isometric embedding, as desired.
\end{proof}

\part{Nonexistence of complete  extensions under curvature conditions}\label{section-nonexistence}\label{part-nonexistence}

This part is devoted to a phenomenological investigation concerning possibile obstructions to the existence of complete Riemannian extensions with controlled curvature. We shall explore techniques of different nature that are sensitive of the topology of the original piece with boundary and that can be used to construct counterexamples to the extension problem.  In this direction it is appropriate to start with the simple observation that, unlike the boundaryless case, compact manifolds with boundary always support Riemannian metrics whose sectional curvature has a prescribed sign.
\begin{lemma}[Gromov]\label{gromov}
Let $M$ be a smooth, $m\geq 2$-dimensional manifold with boundary $\partial M \not=\emptyset$.  Then, there exists a Riemannian metric $g_{+}$ on $M$ such that $\sect_{g_{+}} >0$ and a Riemannian metric $g_{-}$ satisfying $\sect_{g_{-}} <0$.
\end{lemma}
\begin{proof}
Using Theorem \ref{theorem-extension}, we enlarge $M$ past its boundary so to obtain an open Riemannian manifold $N$ containing $M$ as an isometric domain. Next, we apply to $N$ the classical  existence theorem by Gromov, \cite{Gro}, and we restrict the corresponding Riemannian metric to $M$.
\end{proof}
A crucial point is that, with respect to the prescribed Riemannian metric, the boundary could have a wild submanifold geometry. For instance, a convexity condition would immediately force a control on the topology.  In all that follows we will make use of the following convention:
\begin{definition}
Let $(M,g)$ be an $m$-dimensional Riemannian manifold with boundary $\partial M \not=\emptyset$ and outward pointing unit normal $\nu$. The second fundamental form of $\partial M$ is the symmetric bilinear form on $T\partial M$ given by $\mathrm{II}(X,Y) = g(-D_{X}\nu,Y)$ and $\partial M$ is said to be (strictly) convex if the eigenvalues of $\mathrm{II}$ are (strictly) negative. This condition is written as $\mathrm{II} \leq 0$ (resp. $<0$), the inequality being understood in the sense of quadratic forms.
\end{definition}
A result by D. Gromoll, later extended in \cite{Wan}, states that a compact manifold with strictly convex boundary and strictly positive sectional curvature is diffeomorphic to a Euclidean ball, hence it is contractible. In the Ricci curvature setting, we remark the following fact, see \cite{RaSa-JGA}, that can be compared with the results of Section \ref{section-nonexistence-ricci} and Section \ref{section-existence-convex}.
\begin{proposition}
Let $(M,g)$ be a compact, $m\geq 2 $-dimensional Riemannian manifold with boundary $\partial M \not=\emptyset$ and satisfying $\ric_{g} \geq 0$ on $M$. If either \smallskip
 
\noindent $(a)$ $\mathrm{II} \leq 0$ on $\partial M$ and $\ric_{g}>0$ at some $x_{0} \in \inte M$,\smallskip

or \smallskip

\noindent $(b)$ $\mathrm{II} < 0$ on $\partial M$ \smallskip

\noindent then, the first singular homology group $H_{1}(M,\zz)$ is torsion. In particular, the fundamental group $\pi_{1}(M)$ of $M$ cannot contain $\zz^{k}$, $k \geq 1$, as a free or (semi)direct factor.
\end{proposition}
\begin{proof}
There is an Hodge-de Rham theory for manifolds with boundary, \cite{Co-Memoirs}, according to which each de Rham cohomology class of $\mathrm{int}M$ is represented by a harmonic form with Neumann conditions. This follows from an energy minimization procedure in the homology class without boundary restrictions. In particular, given $[\omega] \in H^{1}_{\mathrm{dR}}(\inte M)$, there exists a $1$-form $\xi \in [\omega]$ such that $\Delta_{H} \xi = 0$ and $\xi(\nu) =0$, where $\Delta_{H}$ denotes the Hodge Laplacian. By assuming that the boundary is (weakly) convex $\mathrm{II} \leq 0$, an extension of the Reilly formula to differential forms, \cite{RaSa-JGA}, shows that $\xi$ is parallel. In fact, in case $(b)$, necessarily $\xi = 0$. The same conclusion holds also in case $(a)$ by using Bochner formula. Thus, by de Rham isomorphism, the first real Betti number of $\inte M$ vanishes. Since $M$ has the same homotopy type of $\mathrm{int}M$,  the same holds for $M$ and  the conclusion follows from the universal coefficient theorem and by recalling that the first singular homology group is the Abelianization of the fundamental group.
\end{proof}

The previous discussion leads to the following, a-posteriori obvious, conclusions:
\begin{itemize}
 \item [i)] if we are mainly interested in producing counterexamples to the existence of complete extensions with curvature controls on the base of the topology of the original piece with boundary, it is natural to forget the submanifold geometry of the boundary in order to have much more flexibility. 
 \item [ii)] if we are interested in developing an existence  theory involving the topology of the manifold with boundary, the submanifold geometry of the boundary must be taken under consideration and should play a decisive role.
\end{itemize}

\section{Nonexistence of complete extensions with $\ric \geq C$}\label{section-nonexistence-ricci}
Topological obstructions to the existence of complete metrics with Ricci curvature lower bounds are naturally related to the growth of the fundamental group of the space. Classical tools to detect these obstructions with a nonnegative lower bound are represented by the \v Svarc-Milnor(-Anderson) theory combined with the Bishop-Gromov volume comparison, and the splitting result by Cheeger-Gromoll. Actually, also the harmonic mapping theory  plays a relevant role in this context. On the other hand, as first pointed out by Gromov,  introducing the concept of entropy  in the \v Svarc-Milnor picture allows one to get crucial information even in the case of negative lower bounds. We are going to use these tools to show that, in general, the existence of a complete Riemannian extension with controlled Ricci curvature is prevented by the too much large growth of the fundamental group of the original (compact) manifold with boundary. Concrete examples will be provided for each theoretical result. As a bypass product  we will obtain that, in general, no reasonable Bishop-Gromov type estimates hold for manifolds with bad boundary geometry, as already remarked for instance in \cite{GM}. In particular, these manifolds have no Ricci lower bounds in the sense of the classical singular theory of metric measure spaces.

\subsection{Nonexistence of $\ric \geq 0$ extensions}
Recall that a Riemannian manifold $(M,g)$ with boundary $\partial M \not= \emptyset$ is a length  metric space with respect to its intrinsic distance $d_{g}$. Moreover, if $(M, d_{g})$ is metrically complete then it is a proper geodesic space, i.e., closed metric balls are compact and any couple of points is connected by a minimizing geodesic. Consider the universal Riemannian covering $P:(\tilde M, \tilde g) \to (M,g)$ of $(M,g)$. The Riemannian projection map $P$ gives rise to a metric local isometry $(\tilde M ,d_{\tilde g}) \to (M,d_{g})$. Indeed, using that $P$ preserves the Riemannian lengths, it is easy to see that the length structure of $(\tilde M ,d_{\tilde g})$ is obtained precisely by lifting  via $P$ the length structure of $(M,d_{g})$. Since metric completeness lifts from the base to the covering and conversely, it follows from the Hopf-Rinow theorem that the locally compact length space $(\tilde M ,d_{\tilde g})$ is in fact a geodesic space if and only if so is $(M,d_{g})$. \smallskip

In the following, we summarize some basic facts from the \v{S}varc-Milnor theory for geodesic metric spaces endowed with invariant measures. In particular, it applies to complete Riemannian manifolds with (possibly empty) boundary. 
\begin{theorem}\label{th-SM}
 Let $(X,d)$ be a proper, geodesic, metric space with finitely generated fundamental group $G=\pi_{1}(X)$. Let the universal covering space $\tilde X$ of $X$ be endowed with the lifted length structure $\tilde d$ that makes $P: (\tilde X, \tilde d) \to (X,d)$ a metric local isometry. Then:
\begin{enumerate}
 \item [(a)] $G$ acts freely and properly by isometries on $(\tilde X ,\tilde d)$.
 \item [(b)] Assume that $\tilde \mu$ is a regular Borel measure on $(\tilde X ,\tilde d)$ such that the metric balls have positive measure. If $G$ acts by measure preserving isometries on the metric measure space $(\tilde X ,\tilde d, \tilde \mu)$, then, for any fixed $\tilde x_{0} \in \tilde X$, there exist constants $\alpha,\beta>1$ such that
 \begin{equation}\label{SM-volume}
 |{\mathcal{B}}^{G}_{R}(1)| \leq  \alpha \tilde \mu \left( B^{\tilde X}_{\beta R}(\tilde x_{0}) \right),
 \end{equation}
 for every $R \gg 1$. Here, $ {\mathcal{B}}^{G}_{R}(1)$ denotes the metric ball of $G$ with respect to a fixed finite set of generators.
 \item [(c)] If $X$ is compact, then $G$ is finitely generated and, once it is  endowed with the word metric w.r.t.  a finite set of generators, it is quasi-isometric to $(\tilde X ,\tilde d)$. Moreover \eqref{SM-volume} is completed by the lower estimate
 \[
 \alpha^{-1} \tilde \mu \left(  B^{\tilde X}_{\beta^{-1}R}(\tilde x_{0})\right) \leq | {\mathcal{B}}^{G}_{R}(1)| 
 \]
\end{enumerate}
\end{theorem}
In the setting of Riemannian manifolds without(!) boundary, this result has been sharpened by M. Anderson,  \cite{An-Topology}, via a clever use of Dirichlet-domains of the action. He showed that given  a regular covering $P : ( \tilde M ,\tilde g) \to (M,g)$ with finitely generated deck transformation group $G$, if
\[
\vol B^{\tilde M}_{R}(\tilde x_{0}) \leq A R^{k} \text { and }\vol B^{M}_{R}(P(\tilde x_{0})) \geq B R^{h}
\]
then $G$ grows at most polynomially of order $k-h$. Extending this result to more general spaces (including complete manifolds with boundary) seems nontrivial due to the lack (in general) of nice properties of Dirichlet domains.\\

A direct application of these results yields the next nonexistence criteria.
\begin{theoremA}\label{th-nonexistence-ricc}
Let $(M,g)$ be a complete, $m$-dimensional Riemannian manifold with boundary $\partial M \not= \emptyset$.  Its Riemannian universal covering is denoted by $(\tilde M ,\tilde g)$. Let $G$ be a finitely generated subgroup of  $\pi_{1}(M)$ such that $G$ grows at least polynomially of order $k$. Then the following hold.\smallskip

\begin{itemize}
 \item [(a)] Assume $k \geq m+1$. Then, $(\tilde M ,\tilde g)$ cannot  be extended to a geodesically complete  manifold $( \tilde M',\tilde g')$ satisfying $\ric_{\tilde M'} \geq 0$.
\item [(b)] Assume $k \geq m+1$. If $\pM$ is simply connected, then $(M,g)$ itself has no complete Riemannian  extensions $(M',g')$ satisfying $\ric_{M'} \geq 0$.
\item [(c)] Assume $k=m$. If $\pM$ is simply connected then any Riemannian extension $(M',g')$ satisfying $\ric_{M'} \geq 0$ must be compact. In particular, if $\ric_{M}>0$ at some point then $(M,g)$ has no complete extension $(M',g')$ with $\ric_{M'} \geq 0$. 
\end{itemize}
\end{theoremA}
\begin{proof}
 (a) Since $G$ acts freely and properly by isometries on $(\tilde M ,\tilde g)$ we can consider the (quotient) universal covering projection $Q : \tilde M \to N:=M/G$ where $\pi_{1}(N) \simeq G$ and the smooth manifold $N$ with boundary $\partial N = Q(\partial \tilde M)$ is endowed with the complete metric $h = Q_{\ast}\tilde g$. Now, by contradiction, suppose that the extension $(\tilde M',\tilde g')$ exists. Observe that, having fixed a point $\tilde x_{0} \in \tilde M$, we have the inclusion of intrinsic metric balls $B^{\tilde M}_{R}(\tilde x_{0}) \subseteq B^{\tilde M'}_{R}(\tilde x_{0})\cap \tilde M$. It follows from the Bishop-Gromov volume comparison that $ \vol B^{\tilde M}_{R}(\tilde x_{0}) \leq A R^{m}$. On the other hand, according to \eqref{SM-volume} and recalling that $G$ grows at least polynomially of order $m+1$, we have $\vol B^{\tilde M}(x_{0})\geq B R^{m+1}$ for every $R \gg 1$. Contradiction.\smallskip
 
 (b) By contradiction, suppose that $(M',g')$ is a complete Riemannian extension of $(M,g)$ with $\ric_{M'}\geq 0$. Using collars we can decompose $M'$ as the union $A \cup B$ where $A,B$ are open sets with the homotopy type of $M$ and $M' \setminus M$ respectively and, moreover, $A\cap B$ has the homotopy type of $\partial M$. Since $\partial M$ is simply connected, the Seifert-Van Kampen theorem yields that $\pi_{1}(M') \simeq \pi_{1}(M) \ast H$ for some group $H$. In particular, $\pi_{1}(M')$ contains an isomorphic image of the finitely generated group $G$. This latter acts by isometries on the universal covering space $(\tilde M',\tilde g')$ of $(M',g')$ and gives rise to the regular Riemannian covering projection $Q : (\tilde M',\tilde g') \to (N',h')$ where $N' = \tilde M' /G$ and $h' = Q_{\ast}\tilde g'$. Since $\pi_{1}(N') = G$ is finitely generated and grows at least polynomially of order $m+1$, from \eqref{SM-volume} we know that $\vol B^{\tilde M'}_{R} \geq A R^{m+1}$. But this contradicts Bishop-Gromov because $\ric_{M'} \geq 0$.\smallskip
 
 (c) Let $(M',g')$ be a complete  Riemannian extension of $(M,g)$ such that $\ric_{M'} \geq 0$ on $M'$. As in the previous case, $\pi_{1}(M')$ contains $G$ as a finitely generated subgroup and $G$ has polynomial growth of order $m=\dim M'$.  Consider the Riemannian universal covering projection $Q: (\tilde M',\tilde g') \to (N',h')$ where $N' = \tilde M'/G$. Then, both $(\tilde M',\tilde g')$ and $(N',h')$ are complete Riemannian manifolds with nonnegative Ricci curvature. Moreover $\pi_{1}(N') \simeq G$. We claim that $M'$ is compact. Indeed, suppose the contrary. Since $G$ is a subgroup of $\pi_{1}(M')$, there exists a (possibly not regular) covering projection $N' \to M'$. Therefore $N'$ must be noncompact as well. Fix $\tilde x_{0} \in \tilde M'$ and the corresponding $x_{0} = Q(\tilde x_{0}) \in N'$. By the Calabi-Yau lower volume estimate we have
 \[
 \vol B^{N'}_{R}(x_{0}) \geq C_{1}R
 \]
 for every $R \gg1$ and for some constant $C_{1}=C_{1}(x_{0})>0$. On the other hand, by Bishop-Gromov,
  \[
 \vol B_{R}^{\tilde M'} (\tilde x_{0}) \leq C_{2} R^{m} = C_{2}R^{(m-1)+1}
 \]
 for every $R>0$ and for some dimensional constant $C_{2} = C_{2}(m)>0$. It follows from the Anderson improvement of the \v Svarc-Milnor growth estimate that $G$ grows at most polynomially of order $m-1$. Contradiction. Therefore $M'$ is compact and, hence, $\tilde M'$ contains a line. It follows from the Cheeger-Gromoll splitting theorem that $\tilde M'$ splits isometrically as $\tilde M' \times \rr$, contradicting the fact that $\ric_{ M'}>0$ at some point.
\end{proof}

\begin{remark}
 {\rm
 As a matter of fact, in the assumption of (c), any Riemannian extension $(M',g')$ of $(M,g)$ satisfying $\ric_{M'} \geq 0$ must be compact and Ricci flat. This follows directly from \cite[Theorem 4]{ChGr1}. In particular, if $\sect_{M} \not\equiv 0$, there exists no Riemannian extension satisfying $\ric_{M'} \geq 0$. We have decided to state point (c) solely in terms of Ricci and to provide a \v Svarc-Milnor oriented proof for two reasons: (i) this section is mainly focused on Ricci curvature constraints related to \v Svarc-Milnor theory; (ii) the arguments have the merit to work even in situations where the the volume bounds follow from assumptions not directly related to $\ric \geq 0$.
 }
\end{remark}

\begin{example}
 {\rm
Let $(\Sigma = \mathbb{T}^{2} \setminus D_{\epsilon}, g_{\Sigma})$ be a flat $2$-torus with a small  disc removed and let $(N = \rr^{m}/ \Gamma , g_{N})$ be any closed flat manifold. Then, the Riemannian product $(M = \Sigma \times N, g_{M} = g_{\Sigma}+g_{N})$ is a compact, $(m+2)$-dimensional, flat manifold with boundary $\partial M = \mathbb{S}^{1} \times N$ and fundamental group $\pi_{1}(M) \simeq (\zz \ast \zz) \times \Gamma$. Since $\pi_{1}(M)$ grows exponentially, by Theorem \ref{th-nonexistence-ricc} (a), the Riemannian universal covering $P:(\tilde M ,\tilde g) \to (M,g)$ is a complete, flat, Riemannian manifold with boundary $\partial \tilde M = P^{-1}(\partial M) \not=\emptyset$ and without any complete Riemannian extension $(\tilde M',\tilde g')$ satisfying $\ric_{\tilde g'} \geq 0$.
 }
\end{example}
\begin{remark}
 {\rm
 Observe that, by Theorem \ref{th-SM} (b), the volume growth of the intrinsic balls of the metrically complete, simply connected Riemannian manifold $(\tilde M',\tilde g')$ with boundary is exponential regardless of the fact that it is a flat manifold. This shows that there is no reasonable volume growth comparison for manifolds with uncontrolled boundary geometry. Similar considerations hold for  the next two examples: the Riemannian universal coverings of the compact manifolds we are going to construct are metrically complete manifolds with boundary, with sectional curvature $\geq C>0$ and with polynomial volume growth. In particular, these manifolds are non-compact.
 }
\end{remark}

\begin{example}
 {\rm
 Let $H^{3}_{\rr}$ be the $3$-dimensional Heisenberg group realised as the space of lower triangular matrixes
 \[
\left[ \begin{array}{ccc}
1 & 0 & 0 \\
a & 1 & 0 \\
b & c & 1
\end{array} \right]
\] 
with $a,b,c \in \rr$ and let $H^{3}_{\zz}$ be its natural integral lattice.  Then, $Z = H^{3}_{\rr} / H^{3}_{\zz}$ is a compact $3$-dimensional smooth (Nil)manifold. Its fundamental group $\pi_{1}(Z)$ is isomorphic to $H^{3}_{\zz}$ and, therefore, growths polynomially of order $4$; \cite{Har}. By Seifert-Van Kampen, the same growth property holds for the compact manifold with boundary $Z \setminus D$, where $D \subset Z$ denotes a smooth disk. Now, we construct a $6$-dimensional compact manifold by taking the product
\[
M =  (Z \setminus D) \times Z.
\]
Note that
\[
\partial M \approx \S^{2} \times  Z,
\]
therefore $\partial M$ is not simply connected. On the other hand,
\[
\pi_{1}(M) \simeq H^{3}_{\zz} \times H^{3}_{\zz}
\]
has polynomial growth of order $8 > 6=\dim M$. According to Lemma \ref{gromov} we endow $M$ with a Riemannian metric $g$ of positive sectional curvature $\sect_{M} >0$. By Theorem \ref {th-nonexistence-ricc} (a) we conclude that the universal covering $(\tilde M ,\tilde g)$ of $(M,g)$ cannot be extended to a complete Riemannian manifold of nonnegative Ricci curvature.
}
\end{example}

\begin{example}
{\rm
Let $M = Z \setminus D$ where $Z$ is the compact (Nil)manifold constructed in the previous example and $D$ is a smooth disk. Endow $M$ with a metric $g$ satisfying $\ric_{M} >0$. Since $\partial M \approx \sph^{2}$ is simply connected and $\pi_{1}(M) \simeq H^{3}_{\zz}$ has polymonial growth of order $4$, by Theorem \ref{th-nonexistence-ricc} (b) we conclude that $(M,g)$ cannot be extended to a complete Riemannian manifold of non-negative Ricci curvature.
} 
\end{example}

\begin{example}\label{ex-nonnegativeric-3}
 {\rm
Let $m \geq 3$ and consider the $m$-dimensional compact manifold $M$ with smooth simply connected boundary $\pM \approx \S^{m-1}$ obtained by removing from the ``flat'' torus $\rr^{m}/ \zz^{m}$ a small smooth disk. By the Seifert-Van Kampen theorem, $\pi_{1}(M) \simeq \pi_{1}(\rr^{m} / \Gamma) \simeq \zz^{m}$, therefore $\pi_{1}(M)$ has polynomial growth of order $m$. According to Lemma \ref{gromov} we endow $M$ with a metric $g$ of positive sectional curvature $\sect_{M}>0$. Then, by (c) of Theorem \ref{th-nonexistence-ricc}, $(M,g)$ cannot be extended to a complete manifold $(M',g')$ satisfying $\ric_{M'}\geq 0$.
 }
\end{example}
\subsection{A measure of $\ric_{-}$ of complete extensions}
A compact Riemannian manifold $M$ with simply connected boundary $\partial M \not= \emptyset$ and fundamental group isomorphic to the fundamental group of a flat manifold could have no complete extensions with $\ric \geq 0$; see Example \ref{ex-nonnegativeric-3}. By using some harmonic mapping theory \'a la Schoen-Yau, \cite{EeSa, Ha, ScYa, PRS-book, PiVe-IJM}, we can obtain some more precise information on the negative part of the Ricci curvature of any complete Riemannian extension.\smallskip
\begin{notation}
Given a Schr\"odinger operator $\mathcal{L} = -\Delta_{M}  - a(x)$ on the Riemannian manifold $(M,g)$ we denote by $\lambda_{1}(\mathcal{L})$ the bottom of its spectrum, i.e.,
\[
\lambda_{1}(\mathcal L) = \inf_{\varphi \in C^{\infty}_{c}(M) \setminus \{0\}} \frac{\int_{M} |\nabla \varphi|^{2}+a(x) \varphi^{2}}{\int_{M} \varphi^{2}}.
\]
Moreover, given a real number $\delta \in \rr$ we introduce the modified Schr\"odinger operator
\[
\mathcal{L}_{\delta} = -\Delta_{M} - \delta a(x).
\]
Finally, we use the notation  $a_{-} = -\min(a,0) \in \rr_{\geq 0}$, for any $a \in \rr$.\smallskip

\noindent We say that the Riemannian manifold $(M,g)$ is $\delta$-stable, $\delta >0$, with respect to the Schr\"oedinger operator $\mathcal{L} = -\Delta_{M}-a(x)$ if $\lambda_{1}(\mathcal{L}_{\delta}) \geq 0$.
\end{notation}
With this terminology in mind, we state the following somewhat quantitative result.
\begin{theoremA}\label{prop-usingharmonicmaps}
Let $(M,g)$ be a compact, $m\geq 3$-dimensional Riemannian manifold with boundary $\partial M \not= \emptyset$. Assume also that the following topological properties are satisfied:
\begin{itemize}
\item [(a)] $\partial M $ is simply connected.
\item [(b)]  there exists a non-trivial homomorphism $\varrho: \pi_{1}(M) \to  \Gamma$, where $\Gamma$ is the fundamental group of a compact Riemannian manifold $(N,h)$ with $\sect_{h} \leq 0$.
\end{itemize}
Consider any geodesically complete, noncompact Riemannian extension $(M' , g')$ of $(M,g)$ and set
\[
a(x) = \inf \{ \ric_{M'}(v,v) : v \in T_{x}M' , g'(v,v) =1\}.
\]
Then $(M',g')$ is $\delta$-unstable with respect to $\mathcal{L} = -\Delta_{M'} - a_{-}(x)$, for every $\delta >(m-1)/m$.\\
The same conclusion holds if we replace the assumption that $M'$ is noncompact with the assumption that $\ric_{M'} >0$ at some point.
\end{theoremA}

As alluded to above, the proof is an easy consequence of the harmonic mapping theory developed in \cite{EeSa, ScYa} and of the vanishing theorems  from \cite{PRS-book, PiVe-IJM}. We sketch the arguments for the sake of completeness.

\begin{proof}[Proof (of Theorem \ref{prop-usingharmonicmaps})]
Fix a geodesically complete Riemannian extension $(M',g')$ of $(M,g)$. Since, by the Cartan-Hadamard theorem,  $N$ is aspherical (i.e. its universal covering is contractible), the homomorphism $\varrho$ is induced (up to conjugation) by a continuous map $f: M \to N$ and, since $\partial M$ is simply connected, $f$ can be extended by a constant to all of $M'$. If follows from \cite{ScYa} that there exists a harmonic map $u: M' \to N$ with finite energy $|du| \in L^{2}(M')$ in the homotopy class of $f$. In particular, $u$ and $f$ induce the same homomorphism between fundamental groups.\smallskip

Now, by contradiction, suppose that $\ric_{M'} \geq a(x) g'$ with
\[
\lambda_{1} (-\Delta_{M'} - \delta a_{-}(x)) \geq 0
\]
for some $\delta>(m-1)/m$. \smallskip

Using the vanishing results in \cite{PRS-book, PiVe-IJM} we deduce that $|du| \equiv \mathrm{const}$ and either $a(x) \equiv 0$, i.e., $\ric_{M'} \geq 0$ or $u \equiv \mathrm{const}$. The second possibility cannot occur because, otherwise, the original homomorphism $\varrho$ would be trivial. Therefore $\ric_{M'} \geq 0$. If we assume that $M'$ is non-compact then, by the Calabi-Yau lower volume estimate, we have that $\vol (M') = +\infty$ and, therefore, the constant function $|du|$ cannot be in $L^{2}(M')$. Contradiction. On the other hand, if $M'$ is compact and $\ric_{M'}>0$ at some point, the Weitzenb\"ock formula implies that $|du|=0$ and $u \equiv \mathrm{const}$. Again, this is a contradiction.
\end{proof}

\begin{example}
 {\rm
Let $M$ be the compact, $m \geq 3$-dimensional manifold with simply connected boundary $\partial M \approx \mathbb{S}^{m-1}$ obtained by removing a smooth small disk $D$ from a compact flat manifold $\rr^{m} / \Gamma$, with $\Gamma$ a crystallographic group. Suppose that $M$ is endowed with a Riemannian metric $g$ such that $\ric_{M} \geq 0$. Since $\pi_{1}(M) \simeq \pi_{1}(\rr^{m} / \Gamma) \simeq \Gamma$ and $N = \rr^{m} / \Gamma$ has a flat metric, we can take the obvious isomorphism $\varrho = \mathrm{id}: \pi_{1}(M) \to \Gamma$ and conclude the validity of the following property: if $(M,g)$ is extended to a gedesically complete Riemannian manifold $(M',g')$  with $\ric_{M'} \geq - a(x)$, $a(x) \geq 0$, and $(M',g')$ is $\delta$-stable with respect to $\mathcal{L} = -\Delta - a(x)$, for some $\delta > m/(m-1)$, then $M'$ must be a compact, Ricci-flat manifold.
 }
\end{example}

\subsection{Nonexistence of $\ric \geq -C^{2}$ extensions}
Let $(M,g)$ be a compact Riemannian manifold with (possibly empty) boundary $\partial M$ and let $P : (\tilde M ,\tilde g) \to (M,g)$ be its Riemannian universal covering. The entropy of $(M,g)$ is the number
\[
h(M,g) = \liminf_{R \to +\infty} \frac {\log \vol B^{\tilde M}_{R}(\tilde x_{0})}{R},
\] 
well defined independently of the choice of the base point $\tilde x_{0}$. Similarly, one introduces a notion of entropy in the class of finitely generated groups in order to measure the degree of exponential growth. Let $G$ be a finitely generated group with finite set of generators $\mathcal{S}$. Then, the entropy of $(G,\mathcal{S})$ is the number
\[
h(G,\mathcal{S}) = \liminf_{R\to +\infty} \frac{\log |\mathcal{B}^{G}_{R}(1)|}{R} 
\]
where, we recall, $\mathcal{B}^{G}_{R}(1)$ is the metric ball of $G$ with respect to the word metric induced by the set of generators $\mathcal{S}$. Although a change of the finite set of generators produces quasi-isometric distances, the entropy is not a quasi-isometry invariant and, therefore, it makes sense to define the minimal entropy of the group $G$ as
\[
h(G) = \inf_{\mathcal S} h(G,\mathcal S).
\]
It is a contribution of Gromov to the \v Svarc-Milnor theory that the entropy of a compact manifold is related to the minimal entropy of its fundamental group via the diameter of the space. Although the result is originally stated for a compact manifold without boundary, the proof still works even in the presence of a boundary; see \cite[Theorem 5.16]{Gro-metricstruct}.
\begin{theorem}\label{gromov-entropy}
 Let $(M,g)$ be a compact manifold with (possibly empty) boundary $\partial M$. Then
 \[
 h(\pi_{1}(M)) \leq 2\diam(M,g) h(M,g).
 \]
\end{theorem}

Combining Theorem \ref{gromov-entropy} with Bishop-Gromov volume  comparison we get the following 

\begin{propositionA}\label{SMG}
 Let $(M,g)$ be a compact, $m$-dimensional Riemannian manifold with boundary $\pM \not= \emptyset$. Then, the universal covering $(\tilde M ,\tilde g)$ has no complete Riemannian extensions $(\tilde M',\tilde g')$ satisfying $\ric \geq -(m-1) C^{2}$ for any constant 
\begin{equation}\label{entropy}
0 < C < \frac{h(\pi_{1}(M))}{2 (m-1)\diam(M,g)}.
\end{equation}
 \end{propositionA}
 \begin{proof}
 By contradiction, assume that such an extension $(\tilde M',\tilde g')$ exists. Then, by Bishop-Gromov, having fixed $\tilde x_{0} \in \tilde M \subset \tilde M'$, and using the fact that $B^{\tilde M}_{R}(\tilde x_{0}) \subseteq B^{\tilde M'}_{R}(\tilde x_{0})$, we have
 \[
h(M,g) \leq (m-1) C.
 \]
Using this information into Theorem \ref{gromov-entropy} we conclude
 \[
 h(\pi_{1}(M)) \leq 2 (m-1) \diam(M,g) C
 \]
 and this contradicts \eqref{entropy}.
\end{proof}

Proposition \ref{SMG} suggests that, in general, one can not hope to extend a given manifold with boundary preserving a lower Ricci curvature bound.
This is the content of the following (class of) examples, which are flexible enough to prove that also lower Sectional curvature bounds can not be preserved.

Of course, the situation is different if some further condition on the boundary is prescribed (see for instance Section \ref{section_perelman}).

\begin{example}\label{th_BG}
{\rm
For any real constants $K,\lambda$ and for any dimension $m\geq 2$ there exists an $m$-dimensional Riemannian manifold with boundary $(M,g)$ with constant sectional curvature $\sect_g=K$ such that no extension $(M',g')$ of $(M,g)$ satisfies $\ric_g'\geq \lambda$. 
}
\end{example}

\begin{proof}
Clearly, it is enough to give examples when $\lambda<0$. Accordingly, we set $\lambda=-(m-1)C^2$ for some $C>0$.\smallskip

Let $N=N(C,K)$ be a positive integer to be specified later. Define $N$ points $\{a_i\}_{i=1}^N\subset \R^2$ by $a_i:=(\frac i{4N},0)$ and define the set
$$S_N= \cup_{i=1}^N \partial B_{\frac{i}{4N}}^{\R^2}\left(a_i\right) \subset\R^2.$$ 
Moreover, let $T'_N$ be the $(1/16N)$-neighborhood of $S_N$ in $\R^m$, more precisely $$T'_N:=\left\{x\in\R^m\ : d_{\R^m}(x,(S_N\times \{0_{\R^{m-2}}\}))<\frac 1{16N}\right\}\subset B_{1}^{\R^m}.$$ 
By approximation, $T'_N$ can be homotopically deformed to a submanifold $T_N\subset B_{1}^{\R^m}$ with smooth boundary. Endow $B_1^{\R^m}$ with a metric $g_K$ of constant curvature $K$ and let $(M_N=\tilde T_N,g)$ be the (noncompact) universal Riemannian covering of $(T_N,g_K|_{T_N})$. It is clear by construction that we can suppose that 
\begin{equation}\label{diam_bound}
\operatorname {diam}_{g_K} (T_N)<\delta
\end{equation}
for some constants $\delta>0$ depending on $K$, but independent of $N$.

By construction there is a deformation retraction of $T_N$ onto $S_N$. In particular the
 (equivalence classes of the) loops $\gamma_i:[0,1]\to T_N$ defined by $\gamma_i(t)=\frac{i}{4N}(1+\cos(2\pi t),\sin(2\pi t))\times\{0_{\R^{m-2}}\}$ are a family of generators for $\Gamma=\pi_1(T_N,0)$. Moreover $\pi_1(T_N,0)$ is exactly the free group generated by the $\mathcal S=\{[\gamma_i]\}_{i=1}^N$.

As observed by Gromov, \cite[Example 5.13]{Gro-metricstruct}, the minimal entropy of the free group with N generator is given by $h(\Gamma)=h(\pi_1(T_N))=\log(2N-1)$. Hence, according to Proposition \ref{SMG}, recalling also \eqref{diam_bound}, we get that $(M_N,g)$ admits no Riemannian extensions satisfying  $\ric \geq -(m-1) C^{2}$ provided
\[
N>\frac 12+ \frac12 e^{2(m-1)\delta C}.
\]
\end{proof}

\begin{remark}{\rm
All the fundamental groups of the manifolds $T_N$ involved in the previous example have exponential growths. As a matter of fact, for positive lower curvature bounds, similar techniques permit to construct examples given by coverings of manifolds whose fundamental group has polynomial growth. A simple example in this sense is given by the universal covering of the sphere strip $\{(x,y,z)\in\R^3\ :\ x^2+y^2+z^2=1\text{ and }z\leq 1/2\}$
which has sectional curvature $1$, but is non-compact and hence it does not admit any extension of positively lower bounded curvature (once again by Bishop-Gromov). By the way, taking suitable coverings of large enough finite index, also compact examples can be constructed.}
\end{remark}

\section{Nonexistence of complete extensions with $\sect \leq 0$}

Recall that a  topological space $X$ is said to be $k$-connected, $k\in \nn \cup \{\infty\}$, if its homotopy groups satisfy $\pi_{j}(X) = 0$ for every $j= 0,\cdots ,k$. The connected space $X$ is called aspherical if the vanishing condition $\pi_{j}(X) =0$ holds for every $j \geq 2$.  A classical theorem of J.C. Whitehead implies that, in the setting of $CW$-complexes, $\infty$-connected spaces are contractible. Since the universal covering projection $\tilde X \to X$ induces isomorphisms $\pi_{j}(\tilde X) \simeq \pi_{j}(X)$ for every $j \geq 2$, then aspherical $CW$-complexes are characterised by the property that their universal covering spaces are contractible. This equivalence, in particular, holds at the smooth manifold level. It then follows from the Cartan-Hadamard theorem that every complete Riemannian manifold of non-positive curvature must be aspherical. Since every compact manifold with boundary can be endowed with a (complete) metric of negative curvature, we are naturally led to detect a (sufficiently large) class of smooth compact manifolds with boundary that cannot be realized as a domain inside an aspherical manifold.

We propose two criteria of non-extendibility: one is homological (hence, in principle, easier to apply) and requires that the manifold is simply connected. The other one is homotopical and requires that the sufficiently connected boundary can be capped by a contractible space. According to this program, let us begin by pointing out the following simple obstruction result.
 
\begin{propositionA}\label{prop-nonexistence-negative-sect}
Let $(M,g)$ be a complete, $m(\geq 4)$-dimensional Riemannian manifold with boundary $\partial M \not= \emptyset$.  Assume that, for some $2 \leq k \leq m-2$, the following topological conditions on $M$ and $\partial M$ are satisfied:
\begin{itemize}
 \item [(a)]  $M$ is simply connected and $H_{k}(M;\zz) \not= 0$;
 \item [(b)]  $H_{i}(\pM ; \zz) =0$, $i=k-1, k$.
\end{itemize}
Then, $(M,g)$ has no complete Riemannian extension $(M',g')$ without boundary and satisfying $\sect_{M'} \leq 0$.
\end{propositionA}

\begin{remark}
 {\rm
We stress that  $\pM$ must be non-convex. Otherwise, using some harmonic mapping theory for manifolds with boundary, \cite{Ha-LNM}, we would get that the compact, simply connected manifold $M$ of nonpositive curvature is necessarily contractible. Indeed, every element in $\pi_{k}(M)$, $k \geq 2$, is represented by a harmonic map in its homotopy class. Since $(M,g)$ has non-positive curvature and $\S^{k}$ has a metric of positive curvature, standard vanishing results based on the Weitzenb\"ock formula yield that the harmonic map is constant. Whence, we conclude that $\pi_{j}(M) = 0$, for every $j \geq 0$ and, hence, $M$ is contractible.
More generally, the sectional curvature $\sect_{\pM}$ of $\pM$ with respect to the induced metric must satisfy $\sect_{\pM} > 0$ at some point where $\pM$ is non-convex. Otherwise, as observed by Gromov, \cite{Gr-hyperbolic}, and according to the main result in \cite{ABB-TAMS}, $M$ has non-positive curvature in the sense of Alexandrov. Hence, the Gromov version of the Cartan-Hadamard theorem applies, \cite{AB-EnsMath}, proving that once again $M$ is contractible.
 }
\end{remark}
\begin{remark}
 {\rm
It would be interesting to obtain homological obstructions to a non-positively curved extension of $M$ when this latter space is non-simply connected (maybe with simply connected boundary). In this respect, recall that by passing to the universal covering does not preserve homology groups.
 }
\end{remark}

\begin{proof}
By contradiction, suppose that there exists a complete Riemannian extension $(M',g')$ of $(M,g)$ satisfying $\sect_{M'} \leq 0$. Let $P : (\tilde M',\tilde g') \to (M',g')$ be its Riemannian universal covering and consider the Riemannian manifold
$
 (\tilde M , \tilde g) = ( P^{-1}(M),\tilde g'|_{P^{-1}(M)})
 $
 with boundary
 $
 \partial \tilde M = P^{-1}(\pM).
 $
 Clearly, $(\tilde M ' ,\tilde g')$ is a complete Riemannian extension of $(\tilde M ,\tilde g)$ satisfying $\sect_{\tilde M'} \leq 0$. In particular, by the Cartan-Hadamard theorem, $\tilde M'$ is contractible. Now, the restricted map $P|_{\tilde M} : (\tilde M ,\tilde g) \to ( M,g)$  is still a covering projections and since $M$ is simply connected then $\tilde M$ is a disjoint union of isometric copies of $M$. We still denote with $(\tilde M,\tilde g)$ one of these components. It follows that we can identify $\tilde M \approx M$ and $\partial \tilde M \approx \pM$. Using collars, we  decompose $\tilde M'$ as the union of open sets $A \cup B$, where $A$ has the same homotopy type of $ M$ and $A \cap B$ has the same homotopy type of $\pM$. Applying the Mayer-Vietoris sequence
 \[
\to H_k(\pM;\zz) \to H_k(M;\zz) \oplus H_k(B;\zz) \to H_k(\tilde M';\zz) \to H_{k-1}(\pM;\zz)\to
 \]
 and using the topological assumptions (a), (b) we conclude that there exists an injective homomorphism
 \[
0\not= H_{k}(M;\zz) \hookrightarrow H_{k}(\tilde M' ;\zz).
\]
This contradicts the fact that $\tilde M'$ is contractible.
 \end{proof}
 
\begin{example}
 {\rm
 Let $Z$ be the $m(\geq 4)$-dimensional smooth manifold given by
 \[
 Z = (\S^{m_{1}} \times \S^{n_{1}})\#\cdots\#(\S^{m_{r}} \times \S^{n_{r}}),
 \]
 with $m_{i},n_{i} \geq 2$ and $\sum_{i=1}^r m_{i}+n_{i} = m$. We remove from $Z$ a smooth $m$-dimensional contractible disk $D$ so to obtain the $m$-dimensional compact manifold
 \[
 M = Z \setminus D
 \]
 with smooth boundary $\pM \simeq \S^{m-1}$. Next, using Lemma \ref{gromov}, we endow $M$ with a Riemannian metric $g$ of negative sectional curvature. We claim that $M$ satisfies the topological assumptions of Proposition \ref{prop-nonexistence-negative-sect} and, hence, $(M,g)$ cannot be extended to a complete Riemannian manifold of nonpositve curvature. This follows from the next observations.\smallskip
 
 \noindent $\bullet$ Take the connected sum $Q = M_{1} \# M_{2}$ with $M_{1},M_{2}$ smooth compact manifolds of  dimension $m \geq 3$. Topologically, $Q$ is obtained from $M_{1}$ and $M_{2}$ by removing an $m$-dimensional disk $D^{m}$ from each of these manifolds and identifying their $\S^{m-1}$-boundaries.\smallskip
 
 \noindent $\bullet$ Applying twice the Seifert-Van Kampen theorem we deduce that both $\pi_{1}(M_{j}) \simeq \pi_{1}(M_{j}\setminus D^{m})$ and $\pi_{1}(Q) \simeq \pi_{1}(M_{1}\setminus D^{m}) \ast \pi_{1}(M_{2} \setminus D^{m})$. Therefore, $Q$ is simply connected provided both $M_{1}$ and $M_{2}$ are simply connected.
 
\noindent $\bullet$ Similarly, we see that $Q \setminus D^{m}$ is simply connected provided both $M_{1}$ and $M_{2}$ are simply connected.
 
 \noindent $\bullet$ Using twice the Mayer-Vietoris sequence and that $H_{i}(\S^{m-1};\zz) =0$ for $i\not= 0 ,m-1$ we deduce, for every $k=2,\cdots, m-2$,
 \[
\begin{cases}
 H_{k}(M_{j};\zz) \simeq H_{k}(M_{j} \setminus D^{m};\zz) \\
 H_{k}(Q;\zz) \simeq H_{k}(M_{1} \setminus D^{m};\zz) \oplus H_{k}(M_{2} \setminus D^{m};\zz).
\end{cases}
 \]
 Therefore,  $H_{k}(Q;\zz) \not= 0$ for some $2 \leq k \leq m-2$ provided either $H_{k}(M_{1};\zz) \not=0$ or $H_{k}(M_{2};\zz) \not=0$. \smallskip
 
  \noindent $\bullet$ To conclude that $M$ satisfies the assumptions of Proposition \ref{prop-nonexistence-negative-sect}, we combine the previous observations with the  K\"unneth formula in the absence of torsion: 
  \begin{align*}
  H_{k}(M;\zz) \simeq H_{k}(Z;\zz) &\simeq \oplus_{i=1}^{r}H_{k}(\S^{m_{i}}\times \S^{n_{i}};\zz)\\
  &\simeq \oplus_{i=1}^{r}\oplus_{a+b=k} H_{a}(\S^{m_{i}};\zz) \otimes H_{b}(\S^{n_{i}};\zz).
  \end{align*}
  This shows that $H_{k}(M;\zz) \not=0$ for $2 \leq k=m_{i},n_{i} \leq m-2$.
 }
\end{example}
\begin{example}
 {\rm
 Another family of examples in dimension $m = n+5 \geq 8$ can be obtained by taking the product
\[
M^{m} = ( N^{5} \setminus D^{5} ) \times \S^{n}, \, n \geq 3
\]
  where $N^{5}$ is any of the closed simply connected $5$-manifolds with $H_{2}(N^{5} ;\zz) \not=0$ constructed by D. Barden in \cite{Ba}, and $D^{5}$ is a small $5$-disk inside $N^{5}$. Thus, $M^{m}$ is a simply connected, compact manifold with boundary
  \[
  \pM^{m} \approx \S^{4} \times \S^{n}.
  \]
Using again the K\"unneth formula, and recalling also that $\mathrm{Tor}(\cdot, \zz) = 0$, we see that
 \[
  H_{2}(M^{m};\zz) \simeq H_{2}(N^{5}\setminus D^{5};\zz) \otimes \zz \simeq H_{2}(N^{5} ;\zz) \otimes \zz  = H_{2}(N^{5};\zz)\not= 0.
 \]
and
\[
H_{1}(\pM^{m};\zz) = H_{2}(\pM^{m};\zz) =0.
\]
Therefore, Proposition \ref{prop-nonexistence-negative-sect} applies and gives that $M^{m}$ is not a domain into a complete Riemannian manifold of nonpositive sectional curvature.
 }
\end{example}

Now we consider obstructions of homotopical nature.
\begin{theoremA}\label{th-nonexistence-negative}
 Let $(M,g)$ be a complete, $m$-dimensional Riemannian manifold with boundary $\partial M \not= \emptyset$. Assume that $M$ can be embedded as a domain inside a smooth $m$-dimensional  manifold $N$ without boundary such that:
\begin{itemize}
 \item [(a)] $B = N \setminus M$ is contractible;
 \item [(b)] $\partial M = \partial B$ is $(n-2)$-connected, for some $n \geq 3$
 \item [(c)]  $\pi_{k}(N) \not=0$ for some $2 \leq k \leq n-1$;
\end{itemize}
Then $(M,g)$ has no complete Riemannian extension $(M',g')$ without boundary and satisfying $\sect_{M'} \leq 0$.
\end{theoremA}
We shall use the following intuitive fact.
\begin{lemma}\label{lemma-homotopy}
 Let $N$ be a connected $m$-dimensional manifold given by the union of connected $m$-dimensional manifolds $M$ and $B$ with boundary $\pM = \partial B$. Assume that $B$ is $(n-1)$-connected and that $\pM = \partial B$ is $(n-2)$-connected, for some $n \geq 3$. Then, the inclusion $i: M \hookrightarrow N$ induces isomorphisms $i_{\sharp_{h}}: \pi_{h}(M) \to \pi_{h}(N)$, for $h =0 ,\cdots ,n-2$ and a surjective homomorphism $i_{\sharp_{n-1}} :  \pi_{n-1}(M) \to \pi_{n-1}(N)$. In particular, if $\pi_{n-1}(N) \not=0$ then $i_{\sharp_{n-1}}$ is a nontrivial homomorphism.
\end{lemma}
\begin{proof}
Choose a triangulation of $\partial M = \partial B$ and extend it to a triangulation of $N$. In this way, we can consider $M,B$ as $CW$-subcomplexes of  the $CW$-complex $N$ and $\partial M = \partial B = M \cap B$ as a $CW$ subcomplex of both $M$ and $B$. Let us choose once and for all a point $x_{0} \in B \cap M$ and assume this is the base point in all the homotopical considerations that will follow.

From the long exact homotopy sequence of the pair $( B, \partial B)$:
\[
\to \pi_{k}(B) \to  \pi_{k}(B,\partial B) \to \pi_{k-1}(\partial B) \to \pi_{k-1} (B) \to
\]
since $B$ is $(n-1)$-connected we get
\[
\pi_{k}( B,\partial B) \simeq \pi_{k-1}(\partial B),\, k=0,\cdots,n-1
\]
and since $\partial B$ is $(n-2)$-connected, we deduce that $(B, \partial B)$ is $(n-1)$-connected. Obviously the pair $(M,\pM)$ is $0$-connected. It follows from the homotopy excision theorem, \cite[Theorem 4.23]{Hat-algtop}, that  the inclusion $j : (B, \partial B) \hookrightarrow (N,M)$ induces the isomorphisms
\[
0= \pi_{k}( B,\partial B ) \to \pi_{k}(N,M),\, k=0,\cdots,n-2
\]
and a surjective homomorphism
\[
0 =\pi_{n-1+0}( B, \partial B ) \to \pi_{n-1+0}(N,M),
\]
proving that $(N,M)$ is $(n-1)$-connected. Using this information into the long exact homotopy sequence of the pair $(N,M)$:
\[
\to  \pi_{k}(N,M) \to \pi_{k-1}(M) \to \pi_{k-1} (N) \to \pi_{k-1}(N,M) \to
\]
we conclude that the inclusion map $i : M \hookrightarrow N$ induces  the isomorphisms
\[
i_{\sharp_{h}}:\pi_{h}(M) \simeq \pi_{h}(N), \, h = 0,\cdots, n-2,
\]
and the surjective homomorphism $i_{\sharp_{n-1}}$. This completes the proof of the lemma.
\end{proof}
\begin{proof}
We take a closed collar neighborhood $\mathcal{W} \approx \partial M \times [-1,0]$ of $\partial M$ in $M$ and we set $M_{0} = M \setminus \mathcal{W}$. Then $B_{0} = N \setminus M_{0}$ is contractible because it is a deformation retract of $B$. It follows that there exists a homotopy $H : B_{0} \times [0,1] \to B_{0}$ between $H(\cdot,0) = \mathrm{id}_{B_{0}}$ and $H(\cdot,1) = \epsilon_{q}$, the constant map at $q \in B$. 

 Now, we define a continuous function $F : M=M_{0} \cup \mathcal{W} \to N$ by
 \[
 F(p) =
\begin{cases}
 \mathrm{id}(p) & p \in M_{0}\\
 H(x,t+1) & p=(x,t) \in \mathcal{W}.
\end{cases}
 \]
 Observe that
 \[
 F(p) \equiv q, \, \text{on } \partial M.
 \]
 Therefore, if $(M,g)$ admits a complete Riemannian extension $(M',g')$ satisfying $\sect_{M'} \leq 0$ we can extend $F$ to a continuous map $\tilde F : M' \to N$ by setting
 \[
 \tilde F(p) =
\begin{cases}
 F(p) & p \in M\\
 q & p \in M'\setminus M.
\end{cases}
 \]
 In particular, $\tilde F = \mathrm{id}$ on $M_{0}$. With this preparation, and according to Lemma \ref{lemma-homotopy}, we take a representative $\alpha : S^{k} \to M_{0}$ of a non-trivial class in $\pi_{k}(N)$. Since $M'$ is aspherical, $\alpha$ is homotopically trivial in $M'$. But then $\tilde F \circ \alpha = \alpha$ is homotopically trivial in $N$. Contradiction.
\end{proof}

\begin{example}
 {\rm
The assumptions of Theorem \ref{th-nonexistence-negative} are satisfied by any compact $m(\geq 3)$-dimensional manifold $M$ which is obtained by removing a smooth disk $D^{m}$ from a compact manifold $N$ satisfying $\pi_{k}(N) \not=0$ for some $2 \leq k \leq m-1$. For instance, let $N = \S^{m_{1}} \times \S^{m_{2}}$ with $m_{1},m_{2} \geq 1$, $m_{1}+m_{2}=m \geq 3$, and define the manifold with smooth boundary $M = N \setminus D^{m}$. Then, $\pM \approx \S^{m-1}$ is $(m-2)$-connected and $\pi_{j}(N) \simeq \pi_{j}(\S^{m_{1}}) \times \pi_{j} (S^{m_{2}}) \not=0$ for $2 \leq j=\max(m_{1},m_{2}) \leq m-1$. According to Lemma \ref{gromov}, we endow $M$ with a metric $g$ satisfying $\sect_{M}<0$. Therefore, Theorem \ref{th-nonexistence-negative} applies and gives that $(M,g)$ cannot be extended to a complete Riemannian manifold of non-positive sectional curvature.
 }
\end{example}

\part{Existence of complete extensions under curvature conditions}\label{part_existence}

\section{Extending complete manifolds with compact convex boundary} \label{section-existence-convex}

As alluded to in the previous parts of the paper, the presence of a convexity condition on the boundary implies a control on the topology and helps the existence of a complete Riemannian extension where a given curvature bound is preserved. We are going to illustrate this claim by constructing complete Riemannian extensions both under a lower Ricci or scalar curvature bound and with negative sectional curvature.\smallskip

Recall that, according to our convention, the boundary $\partial M \not=\emptyset$ of $(M,g)$ is \textsl{(strictly) convex} if, with respect to the outward pointing unit normal $\nu$ the second fundamental form at each point of $\partial M$ satisfies $\mathrm{II} \leq 0$ (resp. $<0$) in the sense of quadratic forms.

A first natural question to ask is whether an intrinsic convexity of a manifold $M$ with boundary, i.e. (strictly) convexity of its boundary, implies that the manifold $(M,g)$ at hand can be seen as a convex piece of one of its complete extensions $(M',g')$.
In order to answer  this question, let us collect here below some extrinsic notions of convexity. It is worthwhile to recall that several other sligthly different notions of convexity can be found in the literature.  
\begin{definition}
 The complete Riemannian manifold $(M,g)$ with boundary $\partial M \not=\emptyset$ is said to be:
\begin{itemize}
 \item \textsl{strongly convex} if, for every $p,q \in M$, any geodesic of $M$ connecting $p$ with $q$ is contained in $\inte M$ with the possible exception of the endpoints.
 \item \textsl{domain-strongly-convex} if there exists a complete Riemannian extension $(M',g')$ of $M$ such that $M$ is a strongly convex domain. This means that for every $p,q \in M$, any geodesic of $M'$ connecting $p$ with $q$ is contained in $\inte M$ with the possible exception of the endpoints.
\end{itemize}
 \end{definition}
We then have  the following implications.
\begin{lemma} \label{lem_convexity}
 Let $(M,g)$ be a complete Riemannian manifold with compact boundary $\partial M \not=\emptyset$. Then,
\begin{itemize}
 \item [(a)]If $\partial M$ is strictly convex, then $M$ is domain-strongly-convex;
 \item [(b)] If $M$ is domain-strongly-convex, then it is strongly convex.
 \item [(c)] If $M$ is strongly convex, then $\pM$ is convex.
\end{itemize}
\end{lemma}
\begin{remark}{\rm
Some of the reverse implications fail. More precisely:
\begin{itemize}
\item [(i)] The converse of (a) is not true: $M=\{(x_1,\dots,x_n)\in\R^n: \ \sum x_i^4=1\}$ is domain-strogly-convex, but $\pM$ is not strictly convex.
\item [(ii)] The converse of (c) is not true: $M=[0,1]\times \mathbb T^{n-1}$ is not strongly convex, but $\pM$ is convex. 
\item [(iii)] We do not know if the converse of (b) holds or not.
\end{itemize}}
\end{remark}
\begin{proof}

(a)  By Corollary \ref{th_extension}, we can always consider a  complete Riemannian extension $(N,h)$ of $(M,g)$ so that $\Sigma_{0} = \partial M$ is a compact embedded hypersurface of $N$ with second fundamental form $\II$ with respect to the unit normal $\nu$. For $S$ small enough, the normal exponential map $\exp^{\perp}(s \nu(x))$ is a smooth diffeomorphism on $(-S,S) \times \Sigma_{0}$ and, within the normal tubular neighborhood $ \mathcal{U} = \exp_{\Sigma}^{\perp}((-S,S) \times \Sigma_{0})$, the $s$-coordinate represents the smooth, signed distance function from $\Sigma_{0}$. In particular, because of our choice of $\nu$, we have $s(p) < 0$ for every $p \in \inte M \cap \mathcal{U}$. Moreover, the second fundamental form of the parallel hypersurface
\[
\Sigma_{\bar s} = \{x \in \mathcal{U} : s(x) = \bar s\}
\]
is given by
 \[
\mathrm{II_{\Sigma_{\bar s}}}(X,X) = - \Hess(s) (X,X),
 \]
and, up to take a smaller $S$, $\mathrm{II_{\Sigma_{\bar s}}}$ is strictly negative except for the radial direction.

We are going to adapt the construction of \cite[Theorem 4.1]{PV-convex} so to obtain a complete Riemannian extension $(M',g')$ of $(M,g)$ such that the signed distance function $s : \pM \to \rr$ is smooth and strictly convex on $M' \setminus \inte M$ (except for the radial direction $\nabla s$).
To this end, note that on $(-S,S)\times\pM=\left(\exp_\Sigma^\perp\right)^{-1}(\mathcal U)$ the pulled-back metric writes 
\[
\left(\exp_\Sigma^\perp\right)^{\ast}h(s,x)=ds^2+h_s(x)
\] 
for some metric $h_s$ on $\pM$ which depends on $s$.
For $k>0$, consider the family of metrics $j^{(k)}$ on $\pM$ defined as
\[j^{(k)}(s,x):= k^{-1/2} \sinh(\sqrt{k}s)(g_{\partial M})(x),\]
where $g_\pM$ is the Riemannian metric induced by $g$ on $\pM$.
Consider a smooth partition of unity $\phi_j,\phi_h\in C^\infty((0,+\infty))$ such that
\[
0\leq\phi_h(t)\leq 1,\quad\phi_h|_{(0,S/4]}\equiv 1,\quad\phi_h|_{[S/2,\infty)}\equiv 0,\quad\phi_h' \leq 0
\]
and
\begin{equation*}
\phi_j(t)+\phi_h(t)=1,\quad\forall t\in(0,+\infty).
\end{equation*}
Let 
\[
\sigma(s,x):=\phi_h(s)h_s(x)+\phi_j(s)j^{(k)}(s,x),
\]
and define the metric $g'$ on
\[
M':= M \cup \left((0,+\infty)\times \pM \right)
\]
as
\[
g'=\begin{cases}g&\text{on }M,\\ds^2+\sigma(s,x)&\text{on }(-S,+\infty)\times \pM.\end{cases}
\] 
Note that $(M',g')$ is a well defined $n$-dimensional Riemannian manifold. Reasoning as in \cite{PV-convex}, and adapting Lemma 4.2 therein to this setting, we get that for $k$ large enough the signed distance function $s : (-S,+\infty)\times \pM \subset M' \to \rr$ is smooth and $\Hess(s)(X,X)>0$ whenever $X$ is not parallel to $\nabla s$, whereas $\Hess(s)(\nabla s,\nabla s)=0$.

Now, let $x,y \in M$ and let $\gamma : [0,1] \to M'$ be any constant speed geodesic of $M'$ connecting $x$ and $y$.
Consider the function $f:[0,1]\to\R$ given by 
\[
f(t) = s \circ \gamma(t).
\]
Observe that
\begin{itemize}
 \item [(i)] $f \in  C^{\infty}([0,1])$ is well defined;
 \item [(ii)] $f(0) \leq 0$, $f(1)\leq 0$;
 \item [(iii)] $f'(t) = g\big(\nabla s (\gamma(t)) , \dot{\gamma}(t) \big)$;
\item [(iv)] $f'' (t) = \Hess(s)(\dot{ \gamma}(t) ,\dot{\gamma}(t)) \geq 0$, and $f''(t)>0$ whenever 
$$|g(\dot\gamma(t),\nabla s(\gamma(t)))|<|\dot\gamma(t)||\nabla s(\gamma(t))|.$$
\end{itemize}
In particular $f(t)\leq 0$ for all $t\in[0,1]$, i.e. $\gamma:[0,1]\to M$ is also a geodesic of $M$. Suppose $f(t)\equiv 0$. In this case $g\big(\nabla s (\gamma(0)) , \dot{\gamma}(0) \big)=0$, which in turn implies $f'' (0)>0$, giving a contradiction. Then $f(t_0)<0$ for some $t_0\in (0,1)$. Since $f$ is convex and $f(0),f(1)\leq 0$, then $f(t)<0$ for all $t\in (0,1)$.\\

(b) Suppose by contradiction that there exists a geodesic $\gamma$ of $M$ touching the boundary in a non-extremal point. Namely, since geodesics are locally minimizing, we can suppose that there exists an $\e>0$ and a (unit speed) geodesic $\gamma:[-\e,\e]\to M$ such that $\gamma(0)\in\pM$ and $\gamma$ realizes the distance in $M$ between $\gamma(-\e)$ and $\gamma(\e)$.
Let $M'$ be the extension given by the domain-strong-convexity of $M$. Let $\tilde\gamma:[-\tilde \e, \tilde \e]\to M'$ be any (unit speed) geodesic of $M'$ such that
\begin{enumerate}
\item $\tilde\gamma(-\tilde \e)=\gamma(-\e)$ and $\tilde\gamma(\tilde \e)=\gamma(\e)$,
\item $\tilde\gamma$ realizes the distance in $M'$ between $\gamma(-\e)$ and $\gamma(\e)$.
\end{enumerate}
By assumption, $\tilde\gamma([-\tilde \e,\tilde \e])\subset M$ and $L(\tilde\gamma)\leq L(\gamma)$.
Now, if $L(\tilde\gamma)< L(\gamma)$, then $\gamma$ is not locally minimizing in $M$, contradicting our assumption. Then $L(\tilde\gamma)= L(\gamma)$, which means that $\gamma$ is also a geodesic of $M'$, thus contradicting the domain-strongly-convexity of $M$.\\

(c) Suppose by contradiction that for some point $x\in\pM$ and  vector $v \in T_x\pM$ we have $\II_{\pM}(v,v)> 0$. Let $\gamma:[-1,1]\to N$ be the constant speed geodesic of $N$ such that $ \gamma(0)=x$ and $\dot\gamma(0)=v$. Consider the function $f=s\circ \gamma$ which is well-defined and smooth on $[-\e,\e]$ for some $0<\e\leq 1$. Computing as in the previous implication,  we get that $f(0) = 0$ and $f'(0) = 0$. Moreover 
\[
f'' (0) = \Hess(s)(\dot{\gamma}(0) ,\dot{ \gamma}(0))= \Hess(s)(v ,v)< 0,
\]
so that $f'' (t)<0$ on $[-\e',\e']\setminus \{0\}$ for some $0<\e'\leq \e$.
Thus $\gamma:[-\e',\e']\to M$ is a geodesic of $N$, hence of $M$, with endpoints in $\inte M$ and touching $\pM$. Contradiction.
\end{proof}

\subsection{Existence of complete extensions with $\ric > 0$}\label{section_perelman}
Let $(M,g)$ be a complete Riemannian manifold with compact and strictly convex boundary $\partial M \not= \emptyset$. Assume that $\ric_{M} > 0$. Using Corollary \ref{corollary-localcurvature}, we construct a complete Riemannian extension $(M',g')$ of  $(M,g)$ such that $\ric_{M'} > 0$ and $\partial M'$ is still compact and strictly convex. Now, a result by Perelman, \cite{Per, Wan}, states that the $\Lip$-metric induced by $g'$ on the double $N = M' \cup_{\mathrm{id_{\partial M'}}} M'$ can be smoothen out in an arbitrarily small neighborhood of $\partial M' \subset N$ so to obtain a smooth metric $h$ satisfying $\ric_{h}>0$. Clearly, we can always assume that the perturbed neighborhood of $\partial M' \subset N$ is so small that $(N,h)$ is a Riemannian extension of $(M,g)$. Finally, since the boundary of $M'$ is compact,  $(N,h)$ is obviously complete. We have thus obtained the validity of the following result.
\begin{theoremA}\label{th-perelman}
 Let $(M,g)$ be a compact Riemannian manifold with strictly convex boundary $\partial M \not=\emptyset$ and satisfying $\ric_{M}>0$. Then, there exists a complete Riemannian extension $(N,h)$ of $(M,g)$ without boundary and satisfying $\ric_{N} >0$.
\end{theoremA}
\begin{remark}
 {\rm
A direct application of Theorem \ref{th-perelman} yields a Bonnet-Myers type result for complete manifolds with compact convex boundary and Ricci curvature $\ric \geq k >0$. Actually, using the second variation formula, it is proved in \cite {Li} and \cite{Ge} that a sharp diameter bound can be obtained when $\ric \geq k$, $k \in \rr$, under the assumption that $\partial M$ is compact and strictly mean convex. This is a very  interesting boundary effect.
 }
\end{remark}
\subsection{Existence of complete extensions with $\operatorname{Scal} > 0$}\label{section_GromovLawson}

Since the scalar curvature contains less information than the Ricci or the sectional curvatures, it is natural to expect that Riemannian extensions of manifolds with boundary preserving the positivity of the scalar curvature can be guaranteed under weaker assumptions. In particular, the full strict convexity of $\pM$ is not necessary, as shown by the following result, due to Gromov and Lawson; see  \cite[Theorem 5.7 and Remark 5.8]{GL}.
\begin{theoremA}
 Let $(M,g)$ be a compact Riemannian manifold with boundary satisfying $\operatorname{Scal}_{g}>0$. Suppose that $\pM$ is (strictly) mean-convex, i.e. its mean curvature satisfies $H:=\operatorname{tr}\mathrm{II}<0$ at each point of $\partial M$.
Then, there exists a complete Riemannian extension $(N,h)$ of $(M,g)$ without boundary and satisfying $\operatorname{Scal}_{g}>0$.
\end{theoremA}

As above the extension is constructed on the (differential) double of $M$. The idea of the proof is to consider a tubular $\epsilon$-neighborood of $M$ in $M'\times (-1,1)$, where $M'$ is a local extension of $M$. Using the mean-convexity of the boundary, Gromov and Lawson proved that the (smoothed) surface of this $\epsilon$-neighborhood, endowed with its natural hypersurface metric, still has positive scalar curvature when $\epsilon$ is small enough.
 
\subsection{Existence of complete extensions with $\sect < C^{2}$}

In this section, using a suitable conformal deformation of a local extension, we prove that manifolds with a a compact convex boundary and a nonnegative upper bound of the sectional curvature admit extensions satisfying the same  curvature bound.

\begin{theoremA}\label{th_exis_sect}
Let $(M,g)$ be an $m$-dimensional complete Riemannian manifold with smooth compact boundary $\pM \not=\emptyset$. Assume that $\sect_g< C$ on $M$, for some constant $C>0$, and that $\partial M$ is strictly convex. Then $(M,g)$ has a complete Riemannian extension $(M',g')$ without boundary satisfying $\sect_{g'}<0$. 
\end{theoremA}

\begin{proof}
By (the proof of) Lemma \ref{lem_convexity}, there exists a Riemannian extension $(M',h)$ of $(M,g)$  satisfying the following conditions:
\begin{itemize}
\item [(a)] $M'\setminus M = (0,s_{\ast})\times \pM$ and the metric $h$ on $M'\setminus M$ writes 
\begin{equation}\label{eq_decomp_h}
h(s,x)=ds^2+h_s(x),\end{equation}
$h_s$ being a metric on $\pM$ which varies with $s$.
\item [(b)] The signed distance function $s:  M'\to\R$ from $\pM$ is smooth and strictly convex, except for the radial direction.
\item [(c)] $\sect_{M'} <0$
\end{itemize}
Let $\varphi: (-\infty,s_{\ast}) \to \rr_{\geq 0}$ be a smooth function such that 
\begin{itemize}
\item [(i)] $\varphi=0$ on $(-\infty,0]$;
 \item [(ii)] $\varphi'(t) \geq 0$, $\varphi''(t) \geq 0$;
 \item [(iii)] $\varphi(t)\to +\infty$ as $t\to s_{\ast}$ and $e^\varphi\not\in L^1([0,s_\ast))$.
\end{itemize}
For instance we could choose $\varphi=e^{\tan(\pi(s/s_\ast-1/2))}$ on $(0,s_\ast)$.\\
We define the smooth function $\phi:M'\to\R$ by
\[
\phi(x)=\begin{cases} 0,&\text{if }x\in M,\\ \varphi(s(x)),&\text{if }x\in M'\setminus M.\end{cases}
\]
and we consider the corresponding conformally deformed metric on $M'$:
\[
g'=e^{2\phi}h.
\]
We claim that $(M',g')$ is complete. For integers $n\geq 2$, set $t_n=s_\ast(1-\tfrac1n)$ and consider the smooth, relatively compact exhaustion $\{M'_{n}\}$ of $M'$ given by
\[
M'_{n} = \{x \in M':s(x) < t_{n}\}.
\]
By \eqref{eq_decomp_h},
\[
\mathrm{dist}_{h}(\partial M'_{n} , \partial M'_{n+1}) = t_{n+1} - t_{n},
\]
so that
\[
\mathrm{dist}_{g'}(\partial M'_{n} , \partial M'_{n+1}) > e^{\varphi(t_{n})}(t_{n+1}-t_n) \geq \frac 13e^{\varphi(t_{n})}(t_{n}-t_{n-1})>\frac 13\int_{t_{n-1}}^{t_n}e^{\varphi(t)}dt.
\]
We conclude that, for any divergent path $\gamma:[0,1) \to M'$,
\[
\ell_{g'}(\gamma) \geq \sum \mathrm{dist}_{g'}(\partial M'_{n} , \partial M'_{n+1}) = +\infty,
\]
so that $(M',g')$ is complete (see also Theorem \ref{th_HR}).

Now, we recall the transformation law of the Riemann tensor under conformal changes of the metric, \cite{Be}:
\begin{align*}
R_{g'}(X,Y)Z&=R_h(X,Y)Z+h(X,Z)\nabla_Y\nabla\phi-h(Y,Z)\nabla_X\nabla\phi\\
&+h(Z,\nabla_X\nabla\phi)Y-h(Z,\nabla_Y\nabla\phi)X-(Z\phi)[(X\phi)Y-(Y\phi)X]\\ &+[(X\phi)h(Y,Z)-(Y\phi)h(X,Z)]\nabla\phi \\
&+ h(\nabla\phi,\nabla\phi)(h(X,Z)Y-h(Y,Z)X),
\end{align*}
which implies
\begin{equation}\label{sect-conf}
\sect_{g'}(X\wedge Y) =  e^{-2\phi} \sect_{h}(X \wedge Y) + e^{-2\phi}|X\wedge Y|_{h}^{-2} \mathcal{A}_{h}(X,Y),
\end{equation}
where
\begin{align*} 
\mathcal{A}_{h}(X,Y) &= 2h(X,Y)\Hess \phi(X,Y)\\
&- |Y|^2\Hess\phi(X,X)-|X|^2\Hess\phi(Y,Y)\\
&+|(X\phi)Y-(Y\phi)X|^2-\left(|X|^2|Y|^2-g(X,Y)^2\right)|\nabla\phi|^2.
\end{align*}
Here, $\nabla,\Hess$ and $|\cdot|$ are computed with respect to $h$. Take any point $x\in M'\setminus M$, consider the vector field $W_1=\nabla s$ at $x$, and chose a local frame $\{W_k\}_{k=1}^m$ such that at $x$
\[h(W_i,W_j)=\delta_{ij}\]
Note that $\nabla\phi(x)=\varphi'(s(x))\nabla s(x)$, so that $W_k\phi=0$ for all $k\geq 2$. Moreover
\[
\Hess \phi = \varphi''(s)ds\otimes ds + \varphi'(s)Dd s \geq 0
\]
in the sense of quadratic forms. Accordingly we compute, at $x$,
\begin{align*}
\mathcal{A}_{h}(W_{j},W_{k})=&
2h(W_j,W_k)\Hess \phi(W_j,W_k)\\
-& |W_k|^2\Hess\phi(W_j,W_j)-|W_j|^2\Hess\phi(W_k,W_k)\\
+&|(W_j\phi)W_k-(W_k\phi)W_j|^2-\left(|W_j|^2|W_k|^2-g(W_j,W_k)^2\right)|\nabla\phi|^2\\
<&- \Hess\phi(W_j,W_j)-\Hess\phi(W_k,W_k)\leq 0,
\end{align*}
if $2\leq j < k \leq m$.
Similarly,
\begin{align*}
\mathcal{A}_{h}(\nabla s,W_{k}) =&
2h(\nabla s,W_k)\Hess \phi(\nabla s,W_k)\\
-& |W_k|^2\Hess\phi(\nabla s,\nabla s)-|\nabla s|^2\Hess\phi(W_k,W_k)\\
+&|((\nabla s)\phi)W_k-(W_k\phi)\nabla s|^2-\left(|\nabla s|^2|W_k|^2-g(\nabla s,W_k)^2\right)|\nabla\phi|^2\\
<&h(\nabla s,\nabla\phi)^2 - |\nabla\phi|^2=0
\end{align*}
for all $k=2,\dots,m$. Using these inequalities into \eqref{sect-conf} completes the proof.
\end{proof}

\begin{remark}
 {\rm
 Observe that, actually,
 \[
 \limsup_{x \to \infty\\ \text{ in }M' \setminus M } \sect_{g'} \leq 0.
 \]
Observe also that $M' \setminus M$ has the diffeomorphic type of the cylinder $\partial M' \times(0,+\infty)$. This latter condition, as well as the convexity of the boundary, is intrinsically needed for the construction to work. One may wonder if both the convexity of the boundary and the constraint on the topology of the extended part in Theorem \ref{th_exis_sect} are really needed in order to get a complete extension with $\sect < C$, $C>0$. Some indications are contained in the next section.
 }
\end{remark}

\section{Existence of complete extensions with decaying $\sect$ and general topology of the extended part}
A well known existence result by R. Greene, \cite{Gre}, shows that on every noncompact differentiable manifold, regardless of its topological type, there exists a complete Riemannian metric whose sectional curvature and its derivatives decays to zero at infinity. We are going to use this result to show that a complete manifold with compact boundary can be extended by adding a manifold with (possibly) general topology along which the curvature tensor (and its derivatives) decay to $0$. To begin with, we record the original statement by Greene.

\begin{theorem}[Theorem 1 in \cite{Gre}]\label{Green}
If $\{K_i\}_{i=1}^\infty$ is a  relatively compact exhaustion of a noncompact manifold $N$, $\{s_i\}_{i=1}^\infty$ is a sequence of nonnegative integers and  $\{\e_i\}_{i=1}^\infty$ a sequence of positive numbers, then there is a complete Riemannian metric $h$ on $N$ whose Riemann tensor $R_h$ satisfies
\[\|D^sR_h\|_h(p)\leq \e_i\]
for all nonnegative integers $s\leq s_i$ and for all $p\in \bar K_i\setminus K_{i-1}$, $i> 1$.
\end{theorem}

To prove this result, without loss of generality, it is assumed that each $\partial K_{i}$ is a smooth, compact hypersurfaces so that it has a bicollar neighborhood $\mathcal{W}_{i} \subset N$ of the form  $\partial K_i \times (-1/4,1/4)$. The manifold $N$ is endowed with a Riemannian metric $h_0$ which is the product metric on  $\partial K_i \times [-1/8,1/8] \subset \mathcal{W}_{i}$. This can be done using a partition of unity argument that leave unchanged the original metric on any given domain $\Omega \Subset K_{1}$ but can introduce a large amount of curvature on $K_{1} \setminus \Omega$ and $K_{i} \setminus K_{i-1}$, $i \geq 2$. Then, it is performed a stretching of $h_0$ on each $\bar K_i \setminus K_{i-1}$ with a constant large enough depending on $i$ so to obtain a Riemannian metric $h_1$ with small curvature away from $\partial K_i$.  Finally, in product coordinates of $\mathcal{W}_{i}$, $h_1$ is deformed along the ``radial direction'' in a neighborhood of $\partial K_i$. All of these deformations are performed by keeping the distance between $\partial K_{i+2}$ and $\partial K_{i}$ at most $1$, for every $i$. Whence the completeness of the final metric follows trivially.

We point out that, once the metric on the domain $\Omega$ is prescribed, the above construction can be carry out without changes in the weaker assumption that  $K_{1} \setminus \Omega$  and $\bar K_{i} \setminus K_{i-1}$ are compact, $i \geq 2$.\smallskip

In view of the above arguments, joint with the fact that a Riemannian metric $g$ on a smooth domain $\Omega \subset N$ can be always extended slighlty past  $\partial \Omega$ in $N$, the validity of following extension result is now clear.
\begin{theoremA}
Let $(M,g)$ be an $n$-dimensional complete Riemannian manifold with smooth compact boundary $\pM\not= \emptyset$. Let $Q$ be any smooth manifold with boundary $\partial Q \not=\emptyset$ diffeomorphic to $\partial M$ via a selected map $\eta : \partial M \to \partial Q$ and define the smooth manifold $M' = M \cup_{\eta} Q$.
Then, there exists a complete Riemannian metric $g'$ on $M'$ such that $(M',g')$ is a Riemannian extension of $(M,g)$ and, for every $s\in\N$,
\[
\lim_{x \to \infty \text{ in } M'\setminus M}\|D^sR_{g'}\|_{g'}(p)\to 0.
\]

\end{theoremA}

\section{Existence of complete extensions with $\ric < C$ or $\scal<C$}
Lohkamp proved in \cite{Loh-Annals2,Lohkamp-JDG} that every differentiable manifold with empty boundary admits, for any given $C\in\R$, a Riemannian metric $g$ with $\ric_g<C$.
In view of this remarkable result, one expects that  a lower Ricci curvature bound is less restrictive
than other curvature conditions even in the presence of a nontrivial boundary. It turns out that the method used by Lohkamp goes through local deformations of the metric, and can thus be adapted to lower the Ricci curvature of a given Riemannian metric outside a fixed domain.
This idea, together with Corollary \ref{th_extension}, permits in fact to prove the following

\begin{theoremA}\label{th_Lohkamp}
Let $(M,g)$ be an $n$-dimensional complete Riemannian manifold with smooth nonempty boundary $\pM$. Suppose that $\ric_g< C$, for some real constant $C$. Then $M$ admits an extension $(M',g')$ satisfying $\ric_g'<C$. 
\end{theoremA}

\begin{proof}
In case $M$ is compact, the result follows directly from \cite[Theorem E ]{Loh-Annals2}, while for $M$ non compact we follow \cite[Proposition 2.1]{Lohkamp-JDG}.

Let $(M',h)$ be any complete extension of the noncompact manifold with boundary $(M,g)$, whose existence is guaranteed by Corollary \ref{th_extension}. Since $\ric_{g}< C$ on $M$, then there exists an open collar neighborhood $M\subset\subset  \mathcal{W}\subset M'$ such that $\ric_{h}< C$ in $\mathcal{W}$. For the easiness of notation, throughout this proof for $r> 0$ we set $B_r=B_r^{\R^n}(0)$. Using the paracompactness of $M'$, we find a countable family of diffeomorphisms $f_i: B_6\to M'\setminus M$ such that $\{f_i(B_6)\}_{i=1}^{\infty}$ is locally finite, $\{f_i(B_4)\}_{i=1}^{\infty}$ covers $M'\setminus \mathcal{W}$ and $f_i(B_2)\subset f_{i+1}(B_4\setminus B_3)$. In particular $M'\setminus \mathcal{W}\subset \cup_{i=1}^\infty f_{i}(B_4\setminus B_2)$. 

Let $\chi\in C^\infty(\R, [0, 1])$ be a cut-off function such that $\chi|_{(-\infty,1]}\equiv 0$ and $\chi|_{[2,\infty)}\equiv 1$, and set
$$\begin{cases}F_i(z)=s_i\exp\left(-\frac{d_i}{5-\|f_i^{-1}(z)\|\chi(\|f_i^{-1}(z)\|)}\right),&\text{if}\ \|f_i^{-1}(z)\|< 5\\0&\text{otherwise},
\end{cases}$$
with positive constants $d_i,s_i$ to be chosen later. Define inductively both the metrics
$$
\begin{cases}
h^{0}=h\\
h^{(i+1)}=e^{2F_i}h^{(i)}&i>0.
\end{cases}
$$
and the constants $s_i$ and $d_i$ as follows. Suppose $s_j,d_j$ are given for $j=1,\dots,i$. In particular the metric $h^{(i)}$ is given. By \cite[ Lemma 2.2 ]{Lohkamp-JDG} there exists a $d_{i+1}$ large enough such that
$$
\frac{\ric_{h^{(i+1)}}(X,X)}{h^{(i+1)}(X,X)}-\frac{\ric_{h^{(i)}}(X,X)}{h^{(i)}(X,X)}<\begin{cases}0&on\ f_{i+1}(B_5\setminus B_4)\\-s_{i+1}e^{-d_{i+1}}& on\  f_{i+1}(B_4\setminus B_2)\end{cases}$$
holds for all $s_{i+1}>0$ and for all nonnull vector $X$ with base point in $f_{i+1}(B_5\setminus B_2)$. So we can chose $s_{i+1}$ large enough to make
$$\ric_{h^{(i+1)}}(X,X)<Ch^{(i+1)}(X,X)$$
for all $X \in T_{p'}M' \setminus\{0\}$ with $p' \in f_{i+1}(B_4\setminus B_2)$, hence $p' \in \cup_{j=1}^{i+1}f_{j}(B_4\setminus B_2)$.
Thus the limit metric $h^{(\infty)}$ satisfies the following properties:
\begin{itemize}
\item [-] it is well-defined, since the induction process is locally finite;
\item [-] it has $\ric_{h^{(\infty)}}<C$ by construction;
\item [-] it is complete because it is obtained by a conformal deformation of a complete metric with a conformal factor greater than $1$.
\end{itemize}
This completes the construction of the desired complete Riemannian extension.
\end{proof}

\begin{remark}
{\rm
Theorem E in \cite{Loh-Annals2} can be applied directly also in the non-compact case to get a negatively Ricci curved metric, which however could be incomplete.
}
\end{remark}

\begin{remark}
{\rm
Using Theorem \ref{theorem-extension} instead of Corollary \ref{th_extension} in the proof of Theorem \ref{th_Lohkamp}, one get that for  any smooth $m$-dimensional differentiable manifold $Q$ whose nonempty boundary $\partial Q$ is diffeomorphic to $\pM$, there exists a complete Riemannian extension $(M',g')$ of $(M,g)$ such that $\ric_{g'}<C$ and
$M'\setminus M$ is diffeomorphic to the interior of $Q$.}
\end{remark}

The technique used in the above proof permits to  locally lower the Ricci curvature, hence the scalar curvature, of a given manifold. Accordingly, as a bypass of the proof we get also existence of Riemannian extension preserving an upper scalar curvature bound. Note that a-priori the following result is not a direct consequence of Theorem \ref{th_Lohkamp}.

\begin{theoremA}\label{th_Lohkamp_scal}
Let $(M,g)$ be an $n$-dimensional complete Riemannian manifold with smooth nonempty boundary $\pM$. Suppose that $\scal_g< C$, for some real constant $C$. Then $M$ admits an extension $(M',g')$ satisfying $\scal_g'<C$. 
\end{theoremA}
\medskip

\textbf{Acknowledgement.}
\textit{The authors were partially supported by the} Gruppo Nazionale per l'Analisi Matematica, la Probabilit\`a e le loro Applicazioni (GNAMPA)\textit{. The research of the second author has been conducted as part of the project }Labex MME-DII (ANR11-LBX-0023-01).


\end{document}